\documentclass[a4paper,10pt]{amsart}

{\obeyspaces\global\let =\ }

\font\ttten=cmtt10

\newdimen\outputBaseLineSkip
\newskip\beginOutputSkip
\newskip\endOutputSkip
\def\looserOutput#1{%
  \advance\beginOutputSkip by #1
  \advance\endOutputSkip by #1
}
\def\tighterOutput#1{%
  \advance\beginOutputSkip by -#1
  \advance\endOutputSkip by -#1
}

\def\beginOutput{%
    \par
    \penalty -150
    \penalty -150
    \begingroup
      \def\\{%
          \leavevmode
          \hss
          \endgraf
          \penalty 150
          }
      \ttten
      \parindent = 24pt
      \def\${\char`\$}
      \def\{{\char`\{}
      \def\}{\char`\}}
      \catcode`\_=\the\catcode`a
      \catcode`\^=\the\catcode`a
      \catcode`\#=\the\catcode`a
      \catcode`\~=\the\catcode`a
      \catcode`\&=\the\catcode`a
      \parskip=0pt
      \lineskip=0pt
      \obeyspaces
      }

\def\endOutput{%
    \endgroup
    \par
    \penalty -150
    \penalty -150
    \noindent}

 \usepackage{amsfonts,mathrsfs}
 \usepackage{amsthm}
 \usepackage{amssymb}
 \usepackage{tikz-cd}
 \usepackage{tikz}
 \usepackage{hyperref}
 \usepackage{cleveref}
 \usepackage{url}
 \usepackage{mathtools}
\usepackage{enumitem}

\theoremstyle{definition}
\newtheorem{defin}{Definition}[section]
\newtheorem{ex}[defin]{Example}
\AtEndEnvironment{ex}{\null\hfill\ensuremath{\diamond}}

\newtheorem{remark}[defin]{Remark}

\theoremstyle{plain}
\newtheorem{lem}[defin]{Lemma}
\newtheorem{prop}[defin]{Proposition}
\newtheorem{thm}[defin]{Theorem}
\newtheorem*{thm*}{Theorem}
\newtheorem{cor}[defin]{Corollary}
\newtheorem*{cor*}{Corollary}

\newtheorem*{quest*}{Question}

\newtheorem*{con*}{Conjecture}
\newtheorem{prob}[defin]{Problem}

\usepackage{mathrsfs}

\newcommand{\Res}{\operatorname{Res}}

\newcommand{\conv}{\operatorname{conv}}

\newcommand{\codim}{\operatorname{codim}}

\newcommand{\ord}{\operatorname{ord}}

\newcommand{\adj}{\operatorname{adj}} 

\newcommand{\rk}{\operatorname{rk}}
\newcommand{\crk}{\operatorname{crk}}
\newcommand{\nl}{\operatorname{nl}}
\newcommand{\Span}{\operatorname{span}}
\newcommand{\sing}{\operatorname{sing}}

\newcommand{\vol}{\operatorname{vol}}
\newcommand{\dd}{\text{d}}

\newcommand{\cF}{{\mathcal F}}

\newcommand{\cH}{{\mathcal H}}
\newcommand{\cI}{{\mathcal I}}

\newcommand{\cL}{{\mathcal L}}

\newcommand{\cR}{{\mathcal R}}

\newcommand{\cT}{{\mathcal T}}

\newcommand{\cV}{{\mathcal V}}

\newcommand{\C}{{\mathbb C}}
\newcommand{\R}{{\mathbb R}}
\newcommand{\pp}{\mathbb{P}}

\newcommand{\N}{{\mathbb N}}

\title[Geometry of Adjoint Hypersurfaces]{Geometry of Adjoint Hypersurfaces for Polytopes}

\author{Clemens Brüser}
\address{Technische Universit\"at Dresden, Germany} 
\email{clemens.brueser@tu-dresden.de}
\author{Julian Weigert}
\address{Universit\"at Leipzig and Max Planck Institute for Mathematics in the Sciences, Germany}
\email{julian.weigert@mis.mpg.de}

\begin{document}

\subjclass[2020]{Primary: 14N20, 52B11, 52B40 14Q30, 14J70}

\begin{abstract}
 In this article we prove that the adjoint polynomial of arbitrary convex polytopes is up to scaling uniquely determined by vanishing to the right order on the polytopes residual arrangement. This answers a problem posed by Kohn and Ranestad \cite[Problem 2]{KohnRanestad2019AdjCurves} and generalizes their main Theorem \cite[Theorem 1]{KohnRanestad2019AdjCurves} to non-simple polytopes. We furthermore prove that the adjoint polynomial is already characterized by vanishing to the right order on a zero-dimensional subset of the residual arrangement.
\end{abstract}
\maketitle

\section{Introduction}

The study of adjoint polynomials of convex polygons was first introduced in \cite{Wachspress1975Adjoints} as a means of studying finite element methods. He also gave a framework for studying adjoints of three-dimensional objects, an idea that was later picked up by \cite{Warren1996Adjoints}. Recent approaches have additionally been motivated by mathematical physics, and a widely applicable definition of adjoint polynomials employs the theory of positive geometries \cite{ArkanihamedBaiLam2017PosGeom, Gaetz2020Adjoints, Lam2022PosGeom}.

Let $P\subseteq \R^n$ be a full-dimensional convex polytope which contains the origin as an interior point. We write \begin{align}
\label{eq:polar_dual}
    P^\circ \coloneqq \left\{x\in (\R^n)^\vee\middle| \forall y \in P: \sum_{i=1}^nx_iy_i\geq -1\right\}
\end{align}
for the polar dual of $P$. Furthermore fix an arbitrary triangulation $\cT$ of $P^\circ$. Up to duality Warren defined the adjoint of $P$ as follows \cite{Warren1996Adjoints}.
\begin{defin}
    The \emph{adjoint of $P$} is the polynomial \begin{align*}
        \adj_P(X):=\sum_{\sigma\in \cT}\vol(\sigma)\prod_{v\in V(P^\circ)\setminus V(\sigma)}l_v(X)\in \R[X_0,\ldots,X_n]
    \end{align*}
    where $l_v(X):=\left(X_0+\sum_{i=1}^nv_iX_i\right)$ and $V(P^\circ),V(\sigma)$ denote the vertices of $P^\circ$ and $\sigma$ respectively.
\end{defin}
Warren proved that this definition is independent of the choice of the triangulation $\cT$ \cite[Theorem 4]{Warren1996Adjoints}.
Clearly the adjoint is a homogeneous polynomial of degree $d-n-1$ where $d=|V(P^\circ)|$ is both the number of vertices of $P^\circ$ and the number of facets of $P$. In this article we study the projective hypersurface $A_P:=\{x\in \pp^n\mid \adj_P(x)=0\}$ cut out by this polynomial.

Denote by $\cH_P$ the collection of all projectivized facet hyperplanes of $P$. These are precisely the hyperplanes described by $l_v(X)=0$ for $v\in V(P^\circ)$. The \emph{residual arrangement} $\cR(P)$ is the union of all linear spaces in $\pp^n$ that arise as intersections of facets in $\cH_P$ and do not contain any face of $P$. In \cite[Proposition 2]{KohnRanestad2019AdjCurves} Kohn and Ranestad show that the residual arrangement $\cR(P)$ is contained in the adjoint hypersurface $A_P$. If additionally $\cH_P$ is simple, i.e. if no point $x \in \pp^n$ belongs to more than $n$ hyperplanes in $\cH_P$, then \cite[Theorem 1]{KohnRanestad2019AdjCurves} states that this property also characterizes the adjoint:
The adjoint hypersurface $A_P$ is the unique degree $d-n-1$ hypersurface which contains $\cR(P)$. In this way they provide a geometric characterization of the adjoint $\adj_P(X)$ (up to scaling) via interpolation for polytopes with simple arrangements.  

While slightly perturbing the facets of any polytope will make its hyperplane arrangement simple, many interesting polytopes such as cubes, permutohedra or associahedra, fail to have this property. The results from \cite{KohnRanestad2019AdjCurves} can still be employed by expressing such polytopes as limits of polytopes with simple arrangement. However, this limiting process changes the combinatorial type of the polytopes, making it hard to deduce properties of the adjoint hypersurface of $P$ from those of its approximations. In view of this problem the authors of \cite{KohnRanestad2019AdjCurves} pose the following problem.
\begin{prob}(\cite[Problem 2]{KohnRanestad2019AdjCurves})
     Find a general definition of the adjoint hypersurface of a polytope which does not
depend on the simplicity of its hyperplane arrangement and does not involve limits.
\end{prob}

We provide an answer to this problem by involving singularities of the adjoint hypersurface. Loosely speaking we prove that for any polytope $P$ the adjoint hypersurface can be described as the unique degree $d-n-1$ hypersurface which vanishes at a prescribed set of points and has singularities of a prescribed order at prescribed points. Let $f\in \C[X_0,\ldots,X_n]$ be a homogeneous polynomial and let $x\in \pp^n$ be a point. The order of vanishing of $f$ at $x$, denoted $\mu_x(f)$ is the largest integer $k$ such that applying any composition of less than $k$ partial derivatives to $f$ yields a polynomial that still vanishes at $x$. All smooth points of the variety $\cV(f)$ have order one while singular points have a higher order. Points which don't lie on the variety $\cV(f)$ have order zero. 

Let $\cR_0(P)$ be the finite set of all points in $\pp^n$ that can be described as the unique point in the intersection of some hyperplanes in $\cH_P$. We call $\cR_0(P)$ the \emph{point residual} and notice that - in contrast to the residual arrangement of Kohn and Ranestad - the vertices of $P$ always belong to $\cR_0(P)$. For $x\in \cR_0(P)$ we define the order of $x$ with respect to $P$ as \begin{align*}
    \ord_P(x):=\begin{cases}
        |\{H\in \cH_P\mid x\in H\}|-n &\text{if }x\text{ is a vertex of }P \\
         |\{H\in \cH_P\mid x\in H\}|-n+1 &\text{else}
    \end{cases}
\end{align*}
Our main theorem states that the adjoint hypersurface is the unique degree $d-n-1$ hypersurface which passes through each point in $\cR_0(P)$ with vanishing order at least $\ord_P(x)$.
\begin{thm}
    \label{thm:intro_main_thm}
    The adjoint polynomial $\adj_P(X)$ is the up to scaling unique homogeneous degree $d-n-1$ polynomial satisfying for each $x\in \cR_0(P)$ the vanishing condition \begin{align*}
        \mu_x(f)\geq \ord_P(x).
    \end{align*}
\end{thm}
For special polytopes $P$ and special points $x \in \cR_0(P)$ the inequality above can be strict, we provide some examples in Section \ref{sec: exmpls}.

In order to prove Theorem \ref{thm:intro_main_thm}, we first prove a different version, which is closer to the result of \cite[Theorem 1]{KohnRanestad2019AdjCurves}. This version does not only use points in $\cR_0(P)$, but also includes higher dimensional linear components and postulates vanishing of the adjoint to a certain order along them. For this we extend the definition of the order of $x\in \cR_0(P)$ to linear components of arbitrary dimension that can be cut out by hyperplanes in $\cH_P$, see Definition \ref{def:order_general}. Writing $\cL(P)$ for the collection of all such linear projective spaces our result then reads as follows.
\begin{thm}
\label{thm:intro_main_allcomponents}
     Let $P\subseteq \pp^n$ be a full-dimensional polytope with $d$ facets. There exists an up to scaling unique homogeneous polynomial $f\in \C[X_0,\ldots,X_n]$ of degree $d-n-1$ such that $f$ vanishes along every linear projective space $L\in \cL(P)$ with order at least
 \begin{equation*}
    \mu_L(f)\geq \ord_P(L).
 \end{equation*}
 This polynomial agrees up to scaling with the adjoint $\adj_P(X)$.
\end{thm} 
 Theorem \ref{thm:intro_main_thm} is more effective when computing the adjoint by interpolation, since it only requires checking the correct vanishing of $\adj_P(X)$ at the zero dimensional spaces in $\cL(P)$. In contrast Theorem \ref{thm:intro_main_allcomponents} provides more information about the structure of the adjoint hypersurface $A_P$ and its singularities. More formally we have the following immediate corollary to Theorem \ref{thm:intro_main_allcomponents}. \begin{cor}
     Let $P\subseteq \pp^n$ be a full-dimensional polytope. The singular locus of the adjoint hypersurface contains all linear spaces $L\in \cL(P)$ with $\ord_P(L)\geq 2$. \begin{align*}
         \bigcup \{L\in \cL_P\mid \ord_P(L)\geq 2\}\subseteq \sing(A_P).
     \end{align*} 
 \end{cor}
 The singularities of $A_P$ that can be predicted by Theorem \ref{thm:intro_main_allcomponents} are in general still not exhaustive, see Section \ref{sec: exmpls} for examples. In general describing the singularities of $A_P$ is a challenging open problem, see e.g. \cite{Telen2025Toric} for some results related to singularities of $A_P$.

The article is structured as follows. In \Cref{sec: def} we introduce the necessary notions for the discussion in later sections and state our main results. In \Cref{sec: main_thm_Julian} we prove that the adjoint of a polynomial is characterized by vanishing with the right order in the linear spaces in $\cH_P$. In \Cref{sec: main_cor_Clemens} we prove that the vanishing conditions in positive-dimensional linear spaces are redundant information for this characterization. We conclude with some examples and applications in \Cref{sec: exmpls}.

\section{Acknowledgements}
We thank Rainer Sinn and Mario Kummer for suggesting this problem and the use of canonical forms in the uniqueness part of Theorem \ref{thm:intro_main_allcomponents}. The first author was supported by the DFG grant 502861109. The second author was supported by the SPP 2458 “Combinatorial Synergies”,
funded by the DFG grant 539677510.

\section{Preliminaries and Notation} \label{sec: def}

Throughout this article $P \subseteq \R^n$ denotes a convex full-dimensional polytope in affine space. We identify $\R^n$ with the complement of a projective hyperplane $H_\infty\subseteq \pp^n$ in the complex projective space $\pp^n$ and hence get an embedding $\R^n \hookrightarrow \pp^n$. We do not distinguish between $P$ and its image under this embedding and refer to $P$ as a \emph{projective polytope}. Let $\cH_P$ be the finite collection of \emph{facet hyperplanes} of $P$, i.e. projective hyperplanes which arise by projectivizing a facet of $P$. Unless stated otherwise, when talking about faces of $P$ we include the empty face and the full polytope $P$.

A useful framework for working with hyperplane arrangements such as $\cH_P$ is provided by the theory of matroids. We use the language of matroid theory to keep our proofs concise. However, we emphasize that all matroids appearing in this article are realizable over the reals and therefore all arguments can be restated in terms of linear algebra. For background on matroid theory we refer to \cite{Oxley2011Matroids}.
\begin{defin}
 Denote by $M_P$ the \emph{matroid} corresponding to the hyperplane arrangement $\cH_P$. Its rank function is
 \begin{equation*}
     \rk_{M_P} \colon 2^{\cH_P} \to \N_0;
     S \mapsto \codim_{\pp^n}\left(\bigcap S\right).
 \end{equation*}
\end{defin}

We use the usual notation $\bigcap S:=\bigcap_{L\in S}L$. When $S=\emptyset$ then the intersection above is to be understood as the whole space $\pp^n$ with codimension 0. If the intersection is the empty set, then we use the convention $\codim_{\pp^n}(\emptyset)=n+1$. For a convex polytope $P\subseteq \pp^n$ of full dimension, the intersection of all facet hyperplanes is empty. Hence $M_P$ is a rank $n+1$ matroid on $d\coloneqq |\cH_P|$ elements with no loops and no parallel elements. 

\begin{defin}
 A full-dimensional projective polytope $P\subseteq \pp^n$ with $d$ facets is said to have a \emph{simple arrangement} if $M_P=U_{n+1,d}$ is the uniform matroid of rank $n+1$ on $d$ elements. By definition this means $\rk_{M_P}(S)=\min\{|S|,n+1\}$ for all $S \in 2^{\cH_P}$.
\end{defin}

Notice that having a simple arrangement is a stronger condition than being simple in the classical sense: A polytope is called \emph{simple} if each of its vertices belongs to precisely $n$ facets. Having a simple arrangement clearly implies being simple, but the reverse implication fails. The three-dimensional cube is a standard example of a simple polytope, the arrangement of which is not simple. To see this take $S$ to be the collection of four facet hyperplanes of the cube that come in two pairs of opposite facets. While these four facets don't intersect in a vertex, they do meet in one projective point at infinity and hence $\rk_{M_P}(S)=3\neq 4=\min\{|S|,3+1\}$.

For a point $x\in \pp^n$ let us consider the collection
\begin{equation*}
    \cF_x:=\{H\in \cH_P\mid x \in H\}
\end{equation*}
of all hyperplanes in $\cH_P$ which pass through $x$. This is a flat of the matroid $M_P$. Indeed, for any $H\notin \cF_x$ we have $x \notin H\cap\bigcap\cF_x$ but $x \in \bigcap\cF_x$ and hence
\begin{equation*}
    \rk_{M_P}(\cF_x\cup \{H\})=\codim_{\pp^n}\left( H\cap\bigcap\cF_x\right)> \codim_{\pp^n}\left(\bigcap\cF_x\right)=\rk_{M_P}(\cF_x).
\end{equation*}
In fact every flat of $M_P$ except the full set $\cH_P$ is obtained in this way. For any proper flat $\cF_x$, the intersection of the corresponding hyperplanes in $\cH_P$ gives a projective linear space. We write $\cL(P)$ for the collection of all such spaces: \begin{align*}
    \cL(P):=\left\{\bigcap\cF_x \middle| x \in \pp^n \right\}.
\end{align*}
Since flats are precisely those subsets of maximal size for a fixed rank, it makes sense to keep track of their nullity
\begin{equation*}
    \nl_{M_P}(\cF_x):=|\cF_x|-\rk_{M_P}(\cF_x),
\end{equation*}
which measures how for a set is from being independent in $M_P$. For $L_x:=\bigcap \cF_x \in \cL(P)$ we also write \begin{align*}
    \nl_{M_P}(L_x):=\nl_{M_P}(\cF_x)=|\cF_x|-\codim_{\pp^n}(L_x)
\end{align*} In this language flats are exactly those sets for which the nullity stays constant when adding any element.

Next we define the order of a linear subspace $L\in \cL(P)$ with respect to the polytope $P$. This order will describe a lower bound for the order of vanishing of the adjoint polynomial of $P$ along the linear space $L$.
\begin{defin}
\label{def:order_general}
    Let $L\in \cL(P)$. The \emph{order of $L$ with respect to $P$} is defined to be the number \begin{align*}
    \ord_P(L):=\begin{cases}
        \nl_{M_P}(L) &\text{ if } \dim(L\cap P)=\dim(L) \\
        \nl_{M_P}(L)+1 &\text{ else}.
    \end{cases}
 \end{align*}
\end{defin}

 Notice that if $P$ is simple, then any proper flat of $M_P$ is independent and hence has nullity 0. In this case the condition that $\dim(L\cap P)\neq \dim(L)$ is equivalent to $L$ not containing a non-empty face of $P$, which means $L$ belongs to the residual arrangement as defined by \cite{KohnRanestad2019AdjCurves}. Hence the order defined above recovers the residual arrangement in this case and Theorem \ref{thm:intro_main_allcomponents} specializes to \cite[Theorem 1]{KohnRanestad2019AdjCurves}.

Theorem \ref{thm:intro_main_allcomponents} states that the adjoint $\adj_P(X)$ vanishes along each subspace $L\in \cL(P)$ to at least the order $\ord_P(L)$. We define now what it means for a homogeneous polynomial to vanish along a subvariety of $\pp^n$ to a certain order. For our purposes it is enough to consider linear varieties.

Let $V\subseteq \pp^n$ be a variety and let $\mathcal{I}(V)$ be its vanishing ideal. This means that $\cI(V)$ is the homogeneous radical ideal generated by all homogeneous polynomials in $\C[X_0,\ldots,X_n]$ which vanish on $V$.

\begin{defin}
    Let $f\in \C[X_0,\ldots,X_n]$. The \emph{order of vanishing of $f$ along $V$} is defined as \begin{align*}
        \mu_V(f):=\max \{k \in \N_0 \mid f \in \mathcal{I}(V)^k\}.
    \end{align*} 
    If $V=\{x\}$ is a point, we write $\mu_x(f):=\mu_{\{x\}}(f)$.
\end{defin}

\begin{remark}
\label{rem:vanishing_order}
If $V$ consists of a single point $x \in \pp^n$, then the order of $f$ at $x$ is $k$ if and only if all partial derivatives of order up to $k-1$ of $f$ vanish at $x$ and there is at least one partial derivative of order $k$ which does not vanish at $x$. One can go from points to bigger subvarieties $V$ by noting that $\mu_V(f)=k$ is equivalent to $\mu_x(f)=k$ for almost all $x \in V$. Here "almost all" means that there might by a subvariety of positive codimension in $V$ where $f$ vanishes to an order even higher than $k$.
\end{remark}

Geometrically the vanishing order may be used to detect how singular the hypersurface defined by $f$ is along $V$. If $V=\{x\}$ is a point we have \begin{itemize}
    \item $\mu_x(f)=0$ if $f(x)\neq 0$, i.e. $x\notin V(f)$.
    \item $\mu_x(f)=1$ if $f(x)=0$ but not all first partial derivatives of $f$ vanish at $x$, i.e. $x$ is a smooth point of $V(f)$.
    \item $\mu_x(f)\geq 2$ if $f(x)=0$ and all first partial derivatives of $f$ also vanish at $x$, i.e. $x$ is a singular point of $V(f)$.
\end{itemize}

From the above definition it is easy to see that $\mu_V(\cdot)$ forms a valuation on (the fraction field of) $\C[X_0,\ldots,X_n]$. In other words we have for all $f,g \in \C[X_0,\ldots,X_n]$ \begin{itemize}
    \item $\mu_V(fg)=\mu_V(f)+\mu_V(g)$
    \item $\mu_V(f+g)\geq \min\{\mu_V(f),\mu_V(g)\}$
\end{itemize}

\section{Proof of Theorem \ref{thm:intro_main_allcomponents}} \label{sec: main_thm_Julian}
In the upcoming section we prove Theorem \ref{thm:intro_main_allcomponents}. Using Remark \ref{rem:vanishing_order} and the notation established in the previous section we reformulate Theorem \ref{thm:intro_main_allcomponents} in the following form.
\begin{thm} \label{mainTh}
 Let $P\subseteq \pp^n$ be a full-dimensional polytope with $d$ facets. There exists an up to scaling unique homogeneous polynomial $f\in \C[X_0,\ldots,X_n]$ of degree $d-n-1=\crk{M_P}$ such that $f$ vanishes at every point $x \in \pp^n$ to order
 \begin{equation*}
    \mu_x(f)\geq \ord_P(L_x).
 \end{equation*}
 This polynomial $f$ coincides (up to scaling) with the adjoint $\adj_P$.
\end{thm}

The proof of this result is split into two parts: The existence of a polynomial $f$ with the correct vanishing and singularities and the uniqueness.
For the existence of a polynomial $f$ as in Theorem \ref{mainTh} we will prove that Warren's adjoint of the dual polytope $P^\vee$ has the correct vanishing behavior by constructing for every point $x$ a corresponding triangulation of $P^\vee$. We then use that Warren's adjoint can be computed from any such triangulation as a sum over simplices where we can show that each term individually has the correct vanishing behavior at the point $x$. The proof of uniqueness is more involved. We will use here the theory of positive geometries introduced by \cite{ArkanihamedBaiLam2017PosGeom}. We will show that any $f$ satisfying the above vanishing conditions must up to scaling be the unique numerator of the canonical form of $P$ seen as a positive geometry. This comes down to showing that the above restrictions on $f$ imply a special recursive behavior of a differential $n$-form on $\pp^n$ obtained from $f$. This recursive behavior by definition is unique to the canonical form of the positive geometry $P$.

\subsection{Existence} 
Let $P\subseteq \pp^n$ be a full-dimensional polytope. By a projective change of coordinates we may assume that $P$ is contained in the affine chart $\{X_0=1\}$ and contains the origin $(1:0:\cdots:0)$ as  an interior point. Let $P^\circ\subseteq (\R^n)^\vee \subseteq (\pp^n)^\vee$ be the polar dual of $P$ as defined in (\ref{eq:polar_dual}). Recall that by \cite[Theorem 4]{Warren1996Adjoints} for any triangulation $\cT$ of $P^\circ$ we have \begin{align}
\label{eq:triangulation_formula}
    \adj_P(X)=\sum_{\sigma\in \cT}\vol(\sigma)\prod_{v\in V(P^\circ)\setminus V(\sigma)}l_v(X).
\end{align}
Under the duality of $P$ and $P^\circ$ the vertices $V(P^\circ)$ correspond to the facets of $P$. Hence for each $v\in V(P^\circ)$ we have a facet hyperplane $H_v\in \cH_P$ of $P$. The linear form $l_v(X)=X_0+\sum_{i=1}^nv_iX_i$ is the defining polynomial of $H_v$. 

Fix a point $x \in \pp^n$ and let $I_x=\langle x_jX_i-x_iX_j | i,j \in \{0,\ldots,n\}\rangle$ be its vanishing ideal. We recall the notation $L_x=\bigcap \cF_x\in \cL(P)$ for the linear subspace corresponding to the flat $\cF_x$. Geometrically this is the smallest projective linear space that can be cut out by hyperplanes in $\cH_P$ and contains $x$.  Our goal is to show the following containment.

 \begin{prop}
 \label{prop:existence}
     In the above setup we have \begin{align*}
        \adj_P(X)\in I_x^{\ord_P(L_x)}.
     \end{align*}
 \end{prop}

The definition of the order $\ord_P(L_x)$ suggests to decide between two cases: When the face $L_x\cap P$  of $P$ has dimension $\dim(L_x)$, the order is given by $\nl_{M_P}(L_x)$ and otherwise it is given by $\nl_{M_P}(L_x)+1$. Accordingly we split the proof of Proposition \ref{prop:existence} into two parts. First we show that without any assumptions on $x$ we can always get the vanishing order $\nl_{M_P}(L_x)$ of the adjoint (Lemma \ref{lem:existence_weak_vanishing}). The second and harder part of the proof then consists of showing that if $x$ is such that $L_x$ does not define a face of $P$ of the correct dimension, then we can gain one additional order of vanishing of the adjoint (Lemma \ref{lem:existence_strong_vanishing}).

\begin{lem}
\label{lem:existence_weak_vanishing}
    In the above situation we have 
    \begin{align*}
        \adj_P(X)\in I_x^{\nl_{M_P}(L_x)}.
     \end{align*}
\end{lem}
\begin{proof}
    Let $\cT$ be an arbitrary triangulation of $P^\circ$ so that we can write the adjoint $\adj_P(X)$ as in (\ref{eq:triangulation_formula}). Since $I_x^{\nl_{M_P}(L_x)}$ is an ideal, it is sufficient to show that every term of the sum in  (\ref{eq:triangulation_formula}) individually belongs to $I_x^{\nl_{M_P}(L_x)}$. We hence fix a simplex $\sigma \in \cT$ and consider the polynomial \begin{align*}
        t_\sigma:=\prod_{v\in V(P^\circ)\setminus V(\sigma)}l_v(X).
    \end{align*}
    We show that at least $\nl_{M_P}(L_x)$ factors of this product belong to $I_x$. Then clearly $t_\sigma\in I_x^{\nl_{M_P}(L_x)}$ as claimed.

    Let $v\in V(P^\circ)\setminus V(\sigma)$ and let $H_v\in \cH_P$ be the corresponding facet of $P$. By definition of $I_x$, the containment $l_v(X)\in I_x$ is equivalent to $x\in H_v$. Since the hyperplanes $H\in \cH_P$ that satisfy $x\in H$ are precisely those that satisfy $H\in \cF_x$, we need to show that the set \begin{align*}
        \{v \in V(P^\circ)\setminus V(\sigma)\mid H_v \in \cF_x\}
    \end{align*}
    has cardinality at least $\nl_{M_P}(L_x)=|\cF_x|-\codim_{\pp^n}(L_x)$. Equivalently we need to establish that the cardinality of the set \begin{align*}
       \cH_{\sigma,x}:= \{v \in V(\sigma)\mid H_v\in \cF_x\} 
    \end{align*}
    is bounded from above by $\codim_{\pp^n}(L_x)$. Since $\sigma$ is a simplex, the hyperplanes $H_v$ for $v\in V(\sigma)$ are independent, meaning that the cardinality of $\cH_{\sigma,x}$ agrees with the codimension of $\bigcap_{v\in\cH_{\sigma,x}}H_v$. The latter space contains $L_x$ by definition and hence \begin{align*}
        |\cH_{\sigma,x}|=\codim_{\pp^n}(\bigcap_{v\in\cH_{\sigma,x}}H_v)\leq \codim_{\pp^n}(L_x).
    \end{align*}
\end{proof}

In Lemma \ref{lem:existence_weak_vanishing} we had a lot of leeway: We showed that each term of any possible representation of the form (\ref{eq:triangulation_formula}) of the adjoint vanishes to order at least $\nl_{M_P}(L_x)$. For the more special case when $L_x$ does not define a face of $P$ we will still go term by term in a representation of the form (\ref{eq:triangulation_formula}). However, this time we need to choose the triangulation more carefully and it will depend on the chosen point $x$. In fact this dependence of the triangulation on $x$ is necessary since otherwise the uniqueness part of Theorem \ref{thm:intro_main_allcomponents} would imply that all summands of (\ref{eq:triangulation_formula}) are multiples of each other, which is not true in general.

\begin{lem}
\label{lem:existence_strong_vanishing}
    In the situation above assume additionally that $x$ is such that the face $L_x\cap P$ of $P$ is not of dimension $\dim(L_x)$ of $P$. Then \begin{align*}
        \adj_P(X)\in I_x^{\nl_{M_P}(L_x)+1}.
    \end{align*}
\end{lem}
To prove this lemma we construct an explicit triangulation $\cT_x$ of $P$, which depends on $x$ and which makes the correct vanishing of the adjoint around $x$ manifest. The main ingredient is the following result on triangulations of convex polytopes.
\begin{lem}
\label{lem:costructing_triangulations}
Let $Q\subseteq \R^n$ be a convex polytope, let $C\subseteq Q$ be a face of $Q$ and let $v\in C$ be a fixed vertex. There exists a triangulation $\cT$ of $Q$ such that for every simplex $\sigma \in \cT$ the set of vertices $V(\sigma)\cap V(C)\setminus \{v\}$ is contained in a facet of $C$.
\end{lem}
\begin{proof} 
  An example of such a triangulation is the pulling triangulation with respect to any order of the vertices where the minimal vertex is $v$, see \cite[Section 3]{Manecke2020}. We repeat the construction here for completeness.

  If $Q$ is a simplex, then the pulling triangulation of $Q$ consists only of $Q$ itself.
  Otherwise fix a strict order $\prec$ on the vertices of $Q$ such that $v$ is minimal. \begin{enumerate}
      \item Let $F_1,\ldots,F_s$ be all the facets of $Q$ that do not contain the minimal vertex $\min_\prec(V(Q))$ with respect to $\prec$.
      \item For each $i=1,\ldots,s$ recursively compute the pulling triangulation $\cT_i$ of $F_i$ with respect to the restricted order $\prec|_{F_i}$.
      \item Define the pulling triangulation of $Q$ as \begin{align*}
      \cT:=\{\conv(\sigma,\min_\prec(V(Q)))\mid \sigma \in \bigcup_{i=1}^s\cT_i\}.
  \end{align*}
  \end{enumerate}
Then $\cT$ is a triangulation of $Q$ such that every simplex contains the vertex $v$. For $\sigma'=\conv(\sigma, v) \in \cT$ we have \begin{align*}
      V(\sigma') \cap V(C)\setminus \{v\}= V(\sigma)\cap V(C).
  \end{align*}
  Since $\sigma \in \bigcup_{i=1}^s\cT_i$, there exists a facet $F$ of $Q$ such that $\sigma\subseteq F$ and $v\notin F$. In particular $F$ does not contain $C$ and hence $V(\sigma)\cap V(C)\subseteq F\cap C$ is contained in a facet of $C$.
\end{proof}
\begin{ex}
\label{ex:pulling}
    Let $Q$ be the cube with vertices $e_S=\sum_{i\in S}e_i$ where $S$ ranges over all subsets of $\{1,2,3\}$. Let $v=(1,1,1)$ and let $C=Q\cap \{x_3=1\}$ be the top face of $Q$. For any vertex order starting with \begin{align*}
        v=(1,1,1)\prec (1,1,0) \prec (1,0,0) \prec (0,1,1) \prec \cdots
    \end{align*}
we obtain the triangulation displayed in Figure \ref{fig:pulling_triangulation}.
\begin{figure}
\hspace{0cm}
\begin{tikzpicture}[thick,scale=2]
\coordinate (A1) at (0,0);
\coordinate (A2) at (1,0);
\coordinate (A3) at (0,1);
\coordinate (A4) at (1,1);
\coordinate (B1) at (0.7,0.5);
\coordinate (B2) at (1.7,0.5);
\coordinate (B3) at (0.7,1.5);
\coordinate (B4) at (1.7,1.5);

\node[label=right:{v},circle,fill,inner sep=1pt] at (A4) {};
\node[label=above:{\color{blue} C }] at (A3) {};
\fill[blue,opacity=0.4] (A3) -- (A4) -- (B4) -- (B3) --cycle;
\begin{scope}[dashed,,opacity=0.6]
\draw (A1) -- (B1) -- (B2);
\draw (B1) -- (B3);
\end{scope}
\draw (A3) -- (A4) -- (B4) -- (B3) --cycle;
\draw (A3) -- (A1) -- (A2) -- (B2) -- (B4);
\draw (A2) -- (A4);
\end{tikzpicture}
\hspace{0.25cm}
\begin{tikzpicture}[thick,scale=2]
\coordinate (A1) at (0,0);
\coordinate (A2) at (1,0);
\coordinate (A3) at (0,1);
\coordinate (B1) at (0.7,0.5);
\coordinate (B2) at (1.7,0.5);
\coordinate (B3) at (0.7,1.5);
\coordinate (B4) at (1.7,1.5);

\draw (A1) -- (B1) -- (B2);
\draw (B1) -- (B3);

\draw (A1) -- (A3) -- (B3) -- (B4) -- (B2) -- (A2) --cycle;
\draw[color=red] (A2) -- (B1);
\draw[color=red] (A1) -- (B3);
\draw[color=red] (B1) -- (B4);
\end{tikzpicture}
\hspace{0.25cm}
\begin{tikzpicture}[thick,scale=2]
\begin{scope}[shift={(0.25,0)}]
\coordinate (A1) at (0,0);
\coordinate (A2) at (1,0);
\coordinate (A3) at (0,1);
\coordinate (A4) at (1,1);
\coordinate (B1) at (0.7,0.5);
\coordinate (B2) at (1.7,0.5);
\coordinate (B3) at (0.7,1.5);
\coordinate (B4) at (1.7,1.5);
\coordinate (C1) at (0.7,13/12);
\coordinate (C2) at (0.7,11/12);
    \fill[blue, opacity=0.4] (B4) -- (A4) -- (B3) -- cycle;
    \draw[color=red] (A4) -- (B3);
    \draw[color=red] (B1) -- (A4);
    \draw[color=red] (B1) -- (B4);
    \draw (A4) -- (B4) -- (B3);
    \draw (B3) -- (C1);
    \draw[dashed,opacity=0.6] (C1)--(C2);
    \draw (C2) -- (B1);
    \node[circle,fill,inner sep=1pt] at (A4) {};
\end{scope}

\begin{scope}[shift={(-0.25,0)}]
\coordinate (A1) at (0,0);
\coordinate (A2) at (1,0);
\coordinate (A3) at (0,1);
\coordinate (A4) at (1,1);
\coordinate (B1) at (0.7,0.5);
\coordinate (B2) at (1.7,0.5);
\coordinate (B3) at (0.7,1.5);
\coordinate (B4) at (1.7,1.5);
\fill[blue, opacity=0.4] (A3) -- (A4) -- (B3) -- cycle;
    \draw[color=red] (B3) -- (A4) --(A1);
    \draw (A1) -- (A3) -- (A4);
    \draw (A3) -- (B3);
    
    \draw[color=red,dashed,opacity=0.6] (A1) -- (B3);
    \node[circle,fill,inner sep=1pt] at (A4) {};
\end{scope}

\begin{scope}[shift={(0,0)}]
\coordinate (A1) at (0,0);
\coordinate (A2) at (1,0);
\coordinate (A3) at (0,1);
\coordinate (A4) at (1,1);
\coordinate (B1) at (0.7,0.5);
\coordinate (B2) at (1.7,0.5);
\coordinate (B3) at (0.7,1.5);
\coordinate (B4) at (1.7,1.5);
\coordinate (C1) at (7/32,15/32);
\coordinate (C2) at (189/320,81/64);
    \draw[color=red] (B3) -- (A4) --(A1);
    \draw[color=red] (A1) -- (C1);
    \draw[color=red] (C2) -- (B3);
    \draw[color=red,dashed,opacity=0.6] (C1) -- (C2);
    \draw[color=red] (B1) -- (A4);
    \draw (A1) -- (B1);
    \draw[dashed,opacity=0.6] (B1) -- (B3);
    \node[circle,fill,inner sep=1pt] at (A4) {};
\end{scope}
\begin{scope}[shift={(0.25,-0.25)}]
\coordinate (A1) at (0,0);
\coordinate (A2) at (1,0);
\coordinate (A3) at (0,1);
\coordinate (A4) at (1,1);
\coordinate (B1) at (0.7,0.5);
\coordinate (B2) at (1.7,0.5);
\coordinate (B3) at (0.7,1.5);
\coordinate (B4) at (1.7,1.5);
\coordinate (C1) at (11/10,0.5);
\coordinate (C2) at (41/50,7/10);

    \draw[color=red] (A4) -- (B2);
    \draw[color=red] (C2) -- (A4);
    \draw[color=red,dashed,opacity=0.6] (C2) -- (B1);
    \draw[color=red,dashed,opacity=0.6] (B1) -- (B4);
    \draw (A4) -- (B4) -- (B2);
    \draw (B2) -- (C1);
    \draw[dashed,opacity=0.6] (C1) -- (B1);
    \node[circle,fill,inner sep=1pt] at (A4) {};
\end{scope}

\begin{scope}[shift={(0,-0.5)}]
\coordinate (A1) at (0,0);
\coordinate (A2) at (1,0);
\coordinate (A3) at (0,1);
\coordinate (A4) at (1,1);
\coordinate (B1) at (0.7,0.5);
\coordinate (B2) at (1.7,0.5);
\coordinate (B3) at (0.7,1.5);
\coordinate (B4) at (1.7,1.5);
\coordinate (C1) at (3/4,7/12);
\coordinate (C2) at (3/4,5/12);

    \draw[color=red] (A4) -- (B2);
    \draw[color=red] (A2) -- (C2);
    \draw[color=red,dashed,opacity=0.6] (B1) -- (C2);
    \draw[color= red,dashed,opacity=0.6] (B1) -- (C1);
    \draw[color=red] (C1) -- (A4);
    \draw (A4) -- (A2) -- (B2);
    \draw[dashed,opacity=0.6] (B2) -- (B1);
\node[circle,fill,inner sep=1pt] at (A4) {};
\end{scope}

\begin{scope}[shift={(-0.25,-0.5)}]
\coordinate (A1) at (0,0);
\coordinate (A2) at (1,0);
\coordinate (A3) at (0,1);
\coordinate (A4) at (1,1);
\coordinate (B1) at (0.7,0.5);
\coordinate (B2) at (1.7,0.5);
\coordinate (B3) at (0.7,1.5);
\coordinate (B4) at (1.7,1.5);
    \draw[color=red] (A4) -- (A1);
    \draw[color=red,dashed,opacity=0.6] (A2) -- (B1) -- (A4);
    \draw (A4) -- (A2) -- (A1);
    \draw[dashed,opacity=0.6] (A1) -- (B1);
    \node[circle,fill,inner sep=1pt] at (A4) {};
\end{scope}

\end{tikzpicture}
\caption{Left: A cube $Q$ with distinguished face $C$ (blue) and vertex $v\in C$ from Example \ref{ex:pulling}. Middle: The recursively computed pulling triangulation for the facets not containing $v$. Right: The pulling triangulation of the cube, all simplices contain the vertex $v$.}
\label{fig:pulling_triangulation}
\end{figure}
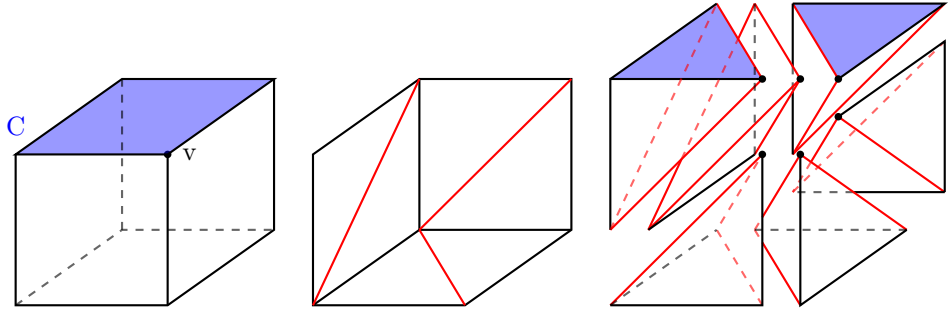
\end{ex}

We turn back to the setting of Lemma \ref{lem:existence_strong_vanishing}. To obtain a suitable triangulation of $P^\circ$ we will apply Lemma \ref{lem:costructing_triangulations} to a special face $C_x$ and a special vertex $v_x\in C_x$ of $P^\circ$.

\begin{proof}[Proof of Lemma \ref{lem:existence_strong_vanishing}]
Consider the (possibly empty) face $\widetilde{C_x}:=L_x\cap P$ of $P$. Since $\widetilde{C_x}$ is a face of $P$, we can write $\widetilde{C_x}$ as the intersection of some hyperplanes in $\cH_P$. By our assumption $\dim(\widetilde{C_x})<\dim(L_x)$ this implies the existence of a facet hyperplane $H_x\in \cH_P$ that contains $\widetilde{C_x}$ but does not contain $L_x$. By duality $\widetilde{C_x}$ corresponds to a face $C_x \subseteq P^\circ$ and the facet $H_x$ containing $\widetilde{C_x}$ corresponds to a vertex $v_x\in V(P^\circ)$ which is contained in $C_x$. Using Lemma \ref{lem:costructing_triangulations} we construct a triangulation $\cT$ of $P^\circ$ such that for every simplex $\sigma\in \cT$, the set $V(\sigma)\cap V(C_x)\setminus \{v_x\}$ is contained in a facet of $C_x$. Using the triangulation $\cT$, we show that every term of (\ref{eq:triangulation_formula}) individually belongs to the ideal $I_x^{\nl_{M_P}(L_x)+1}$. 

Fix $\sigma \in \cT$ and consider the polynomial \begin{align*}
        t_\sigma:=\prod_{v\in V(P^\circ)\setminus V(\sigma)}l_v(X).
    \end{align*}
    We show that at least $\nl_{M_P}(L_x)+1$ factors of this product belong to $I_x$. Then clearly $t_\sigma\in I_x^{\nl_{M_P}(L_x)+1}$ as claimed.

Since a factor $l_v(X), v \in V(P^\circ)$ belongs to $I_x$ precisely when the corresponding facet $H_v\in \cH_P$ belongs to $\cF_x$, we need to show that the set \begin{align*}
        \{v \in V(P^\circ)\setminus V(\sigma)\mid H_v \in \cF_x\}
    \end{align*}
    has cardinality at least $\nl_{M_P}(L_x)+1=|\cF_x|-\codim_{\pp^n}(L_x)+1$. Equivalently we need to bound the cardinality of the set \begin{align*}
       \cH_{\sigma,x}:= \{v \in V(\sigma)| H_v\in \cF_x\} 
    \end{align*}
    from above by $\codim_{\pp^n}(L_x)-1$. Since $H_{v_x}\notin \cF_x$, by definition of the triangulation $\cT$ the vertices in \begin{align*}
        \cH_{\sigma,x}\subseteq V(\sigma)\cap V(C_x)\setminus \{v_x\}
    \end{align*}
    belong to a facet $F$ of $C_x$. On the other hand since $\widetilde{C_x}$ is the largest face of $P$ contained in $L_x$, the set of vertices $\cF_x^\vee:=\{v\in V(P^\circ)|H_v\in \cF_x\}$ is not contained in any facet of $C_x$. This means that the inclusion
    \begin{align*}
    \Span(\cH_{\sigma,x})\subsetneq \Span(\cF_x^\vee) = L_x^\vee
    \end{align*}
    has to be strict since the left hand side is contained in the span of $F$ but the right hand side is not. Now since the vertices in $\cH_{\sigma,x}$ are contained in $V(\sigma)$, they are affinely independent and so we conclude \begin{align*}
        |\cH_{\sigma,x}|= \dim \Span(\cH_{\sigma,x})+1 < \dim L_x^\vee +1 = \codim_{\pp^n}(L_x).
    \end{align*}
    Hence $ |\cH_{\sigma,x}|\leq\codim_{\pp^n}(L_x)-1$ as claimed.
    
\end{proof}
\begin{proof}[Proof of Proposition \ref{prop:existence}]
    Pick $x\in \pp^n$. If $L_x\cap P$ is not a face of dimension $\dim(L_x)$ in $P$ then $\ord_P(L_x)=\nl_{M_P}(L_x)+1$ and the claim follows from Lemma \ref{lem:existence_strong_vanishing}. Otherwise $\ord_P(L_x)=\nl_{M_P}(L_x)$ and the claim follows from Lemma \ref{lem:existence_weak_vanishing}.
\end{proof}

\subsection{Uniqueness}

For the proof of the uniqueness part we will make use of the theory of positive geometries, which was developed by Arkani-Hamed, Bai and Lam in \cite{ArkanihamedBaiLam2017PosGeom}. We only recall the background needed for our argument. We also invite the reader unfamiliar with this beautiful topic to take a look at the introductory articles \cite{Lam2022PosGeom},\cite{Ranestad2025Whatis}.
\begin{defin}
\label{defi_positive_geom}
    Let $X$ be an irreducible complex projective variety of complex dimension $n$ and let $X_{\geq 0}\subseteq X(\R)$ be a semialgebraic subset of the real points of $X$ of real dimension $n$ together with an orientation. The pair $(X,X_{\geq 0})$ is called a \emph{positive geometry} if there exists a unique differential $n$-form $\Omega(X_{\geq 0})$ on $X$, called canonical form, satisfying the following recursive property. \begin{itemize}
        \item If $n=0$ then $X=X_{\geq 0}$ has to be a point and $\Omega(X_{\geq 0})=\pm 1$ where the sign depends on the orientation of $X_{\geq 0}$.
        \item For $n>0$ we require $\Omega(X_{\geq 0})$ to have at most simple poles along the boundary components of $X_{\geq 0}$ and no singularity elsewhere. Furthermore we require each boundary component of $X_{\geq 0}$ to be a positive geometry with canonical form given by the residue of $\Omega(X_{\geq 0})$ along the boundary component (see Definition \ref{def:residue} below).
    \end{itemize}
\end{defin}

The easiest class of examples of positive geometries are polytopes. For any full dimensional projective polytope $P$ as above, the pair $(\pp^n,P)$ is a positive geometry. In this case the canonical form can be given explicitly for all choices of $P$ by referring to the adjoint. In fact by \cite{Gaetz2020Adjoints} the canonical form of $(\pp^n,P)$ is (up to sign) given by 
 \begin{align}
 \label{eq:caninical_poly}
    \Omega(P):=\frac{\adj_P}{\prod_{H\in \cH_P}l_H} \omega_{\pp^n}
\end{align}
where $l_H=l_{v_H}=X_0+\sum_{i=1}^nv_iX_i$ is the linear form corresponding to the vertex $v_H\in V(P^\circ)$ of $P^\circ$ corresponding to $H$ under duality and $\omega_{\pp^n}$ is defined as \begin{align*}
    \omega_{\pp^n}:=\sum_{i=0}^n(-1)^iX_i\dd X_0\wedge\cdots\wedge \widehat{\dd X_i} \wedge \cdots \wedge \dd X_n.
\end{align*}
\begin{remark}
    \begin{enumerate}
        \item We usually omit the orientation when talking about polytopes as positive geometries. Changing the orientation simply multiplies $\Omega(P)$ by $(-1)$.  It is important to choose the orientation along the facets of $P$ in a compatible way when taking residues. To get really equality in (\ref{eq:caninical_poly}), not just up to sign, one needs to pick the orientation that makes $\Omega(P)$ positive on the interior of $P$.
        \item Formally $\omega_{\pp^n}$ can be seen as an $n$-form on $\C^{n+1}$, but it does not descend to a well-defined differential form on $\pp^n$ since it is not $0$-homogeneous. However, the volume forms (i.e. differential $n$-forms) on $\pp^n$ are precisely the forms $g(X)\omega_{\pp^n}$ for a homogeneous rational function $g(X)$ of degree $-n-1$, see e.g. \cite[Appendix C]{ArkanihamedBaiLam2017PosGeom}. In particular $\Omega(P)$ is a well-defined volume form on $\pp^n$. When restricting to the affine chart $X_0=1$ we get the standard volume form on $\C^n$. \begin{align*}
            \omega_{\pp^n}|_{X_0=1}=\dd X_1\wedge \cdots \wedge \dd X_n.
        \end{align*}
    \end{enumerate}
\end{remark}
\begin{defin}[Linear residues]
    \label{def:residue}
    Let $g(X)\in \C(X_0,\ldots,X_n)$ be a homogeneous rational function of degree $-n-1$ and consider the differential form $g(X)\omega_{\pp^n}$. Let $H\subseteq \pp^n$ be a hyperplane cut out by an equation $l_H=\sum_{i=0}^na_iX_i=0$ where $a_n=1$. We identify $H\cong \pp^{n-1}$ via $(X_0:\cdots:X_n)\mapsto (X_0:\cdots:X_{n-1})$. Suppose $g(X)\omega_{\pp^n}$ has a simple pole at $H$, i.e. for suitable rational functions $\widetilde{p}(X), q(X)\in \C(X_0,\ldots,X_n)$ of degree $-n$ we can write \begin{align}
    \label{eq:residue_form}
        g(X)\omega_{\pp^n}&= \left(\widetilde{p}(X)\sum_{i=0}^{n-1}(-1)^iX_i\dd X_0\wedge \cdots \wedge \widehat{\dd X_i}\wedge \cdots \wedge \dd X_{n-1}\right)\wedge \frac{\dd l_H}{l_H}\\& \nonumber+ q(X)\dd X_0\wedge \cdots \wedge \dd X_{n-1}
    \end{align}
    where $q(X),\widetilde{p}(X)$ have no poles around $H$.
    Let $p(X):=\widetilde{p}(X)|_H\in \C(X_0,\ldots,X_{n-1})$ be the rational function obtained by substituting $X_n=-\sum_{i=0}^{n-1}a_iX_i$ in $\widetilde{p}(X)$. We define the \emph{residue of $g(X)\omega_{\pp^n}$ along $H$} to be \begin{align*}
        \Res_H(g(X)\omega_{\pp^n}):=p(X)\omega_{\pp^{n-1}(X)}.
    \end{align*}
    We set the residue to zero if $g(X)\omega_{\pp^n}$ either has no pole around $H$ or one of order at least two.
\end{defin}
\begin{ex}
    Let $n=1$ and consider the line segment $P=\{(1:x)| a\leq x \leq b\}$ for $a\leq b\in \R$. This polytope is a positive geometry with canonical form \begin{align*}
        \Omega([a,b])=\frac{b-a}{(X_1-aX_0)(bX_0-X_1)}\omega_{\pp^1}.
    \end{align*}
    This differential form has a simple pole around the hyperplane $H=\{X_1-aX_0=0\}$ and so we can write \begin{align*}
        &\left(\frac{(b-a)X_0}{bX_0-X_1}\right)\wedge\left( \frac{\dd (X_1-aX_0)}{X_1-aX_0}\right)+ \frac{a-b}{bX_0-X_1}\dd X_0\\ &= \frac{b-a}{(X_1-aX_0)(bX_0-X_1)}(X_0\dd X_1-X_1\dd X_0)= \Omega([a,b]) 
    \end{align*}
    Here the wedge product in the first line is just multiplication since the term to the left of it is a $0$-form. Hence in (\ref{eq:residue_form}) we have $\widetilde{p}(X)=\frac{(b-a)X_0}{bX_0-X_1}$ and $q(X)=\frac{a-b}{bX_0-X_1}$. The residue around $H=\{X_1-aX_0=0\}$ is obtained by restricting $\widetilde{p}$ to $H$: \begin{align*}
        \Res_H(\Omega([a,b])= \frac{(b-a)X_0}{bX_0-aX_0}\omega_{\pp^0}= 1.
    \end{align*}
    Similarly the residue along $bX_0-X_1=0$ is $-1$. These two constants are the two canonical forms at the two vertices of the line segment. They come with distinct signs since any orientation of the line segment restricts to opposite orientations on its two endpoints.
\end{ex}
\begin{remark}
    Some readers might be more familiar with the dehomogenized version of the residue operator, which is obtained by working in affine local coordinates. We prefer the homogeneous version because even when the polytope $P$ is completely contained in an affine chart, its residual arrangement is not contained in the same chart if $P$ has parallel facets. There are two possible ways of dealing with this issue: 
    
    First we could change coordinates to a different affine chart which avoids $P$ and does not contain any of the linear spaces $L_x=\bigcap \cF_x$ entirely. This is always possible and reduces all arguments to the affine case. Consider for example the unit square with vertices $(0,0),(0,1),(1,0),(1,1)\in \R^2\subseteq \pp^2$. The two residual points of this square lie at infinity since they come from intersecting parallel facets. The projective hyperplane $x+y+z=0$ avoids the square and also avoids the two residual points, hence after the projective coordinate change $(x:y:z)\mapsto(x:y:x+y+z)=:(a:b:c)$, the square and its entire residual arrangement is contained in the affine chart $c=1$. When identifying this chart with $\R^2$ via the map $(a:b:c)\mapsto (\frac{a}{c},\frac{b}{c})$, the square gets mapped to the quadrilateral with vertices $(0,0),(0,\frac{1}{2}),(\frac{1}{2},0),(\frac{1}{3},\frac{1}{3})$. 

    The second possibility is to simply work in projective coordinates all along, not distinguishing the hyperplane at infinity at all. Then instead of working with local affine coordinates, we should consider the canonical forms and the residue operator in global homogeneous coordinates. For this article we prefer to stick to the second option since it requires no choice of a hyperplane at infinity.
\end{remark}

For the rest of this section we fix a full-dimensional convex polytope $P\subseteq \pp^n$ with $d$ facets as in Theorem \ref{thm:intro_main_allcomponents}.

Consider a polynomial $f\in \C[X_0,\ldots,X_n]_{d-n-1}$ and assume that it satisfies $\mu_L(f)\geq \ord_P(L)$ for all linear subspaces $L\in \cL(P)$. As noted in Remark \ref{rem:vanishing_order} this implies $\mu_x(f)\geq \ord_P(L_x)$ for all $x \in \pp^n$, where we recall that $L_x=\bigcap\cF_x$ is the intersection of all facet hyperplanes containing $x$. 
We define 
\begin{align*}
     \Omega_f(P):=\frac{f}{\prod_{H\in \cH_P}l_H}\omega_{\pp^n}.
 \end{align*}
 The proof of the uniqueness part of Theorem \ref{thm:intro_main_allcomponents} consist of showing that $\Omega_f(P)$ satisfies the same recursive property as the canonical form $\Omega(P)$ and then deducing that $f$ and $\adj_P$ must agree up to scaling.

 Let us set up some notation for passing to a facet of $P$. Fix one facet hyperplane $H\in \cH_P$ and let $l_H=\sum_{i=0}^na_iX_i$ be its defining linear form. By possibly reordering the variables we may assume $a_n\neq 0$ and hence we can rescale $l_H$ such that $a_n=1$. We identify $H\cong \pp^{n-1}$ via projection to the first $n$ homogeneous coordinates. For a rational function $g\in \C(X_0,\ldots,X_n)$ we write $g|_H\in \C(X_0,\ldots,X_{n-1})$ for the rational function on $\pp^{n-1}$ obtained by substituting $X_n=-\sum_{i=0}^{n-1}a_iX_i$.
 
 We write $P':=P\cap H$ for the polytope in $\pp^{n-1}$ obtained by mapping the facet corresponding to $H$ under the isomorphism $H\cong \pp^{n-1}$. This is again a full-dimensional convex polytope and we write $d'$ for its number of facets. For a facet $G\in \cH_P$ we write $G|H$ if $G\cap H \cap P$ is a codimension 2 face of $P$. This simply means that the corresponding vertices $v_H,v_G\in V(P^\circ)$ on the dual side are connected by an edge. Otherwise we write $G\nmid H$. By convention both of the expressions $H|H$ and $H\nmid H$ are false. The hyperplane arrangement of $P'$ is then given by \begin{align*}
     \cH_{P'}=\{H\cap G \mid G|H\}.
 \end{align*}
 Notice that in a convex polytope it is impossible for three facet hyperplanes to intersect in a linear space of codimension 2. Hence each facet hyperplane is uniquely determined by its intersection with $H$ and we can therefore consider $\cH_{P'}$ as a subset of $\cH_P$. For a point $x \in H=\pp^{n-1}$ we write $\cF_x'=\{G\in \cH_{P'}| x \in G\}$ for the corresponding flat of $M_{P'}$. As before we write $L_x=\bigcap \cF_x$ and we also write $L_x'=H\cap \bigcap \cF_x'$

\begin{lem}
\label{crucial_Lemma}
    In the above setting  assume further that $f$ does not vanish identically along $H$, i.e. $f|_H\neq 0$. Then $\Omega_f(P)$ has a simple pole along $H$, so we can write \begin{align*}
        \Res_H(\Omega_f(P))= \frac{f'}{\prod_{G|H}l_G|_H} \omega_{\pp^{n-1}}
    \end{align*}
    for some non-zero rational function $f'\in \C(X_0,\ldots,X_{n-1})$. Furthermore $f'$ is a polynomial of degree $\deg(f')=d'-(n-1)-1$ and for each point $x \in H$ we have $\mu_x(f')\geq \ord_{P'}(L_x')$.
\end{lem}
\begin{proof}
    Since $f$ does not vanish identically along $H$ we have \begin{align*}
        \Omega_f(P)&=\left(\frac{f}{\prod_{G\in \cH_P\setminus \{H\}}l_G}\sum_{i=0}^{n-1} (-1)^iX_i\dd X_0\wedge \cdots \wedge \widehat{\dd X_i}\wedge \cdots \wedge \dd X_{n-1}\right)\wedge \frac{\dd l_H}{l_H} \\
        &+\frac{(-1)^nf}{\prod_{G\in\cH_P\setminus \{H\}}l_G}\dd X_0\wedge \cdots \wedge \dd X_{n-1}.
    \end{align*}
    It follows that the residue along $H$ is \begin{align*}
        \Res_H(\Omega_f(P))=\left(\frac{f}{\prod_{G\in \cH_P\setminus \{H\}}l_G}\middle)\right|_{H}\omega_{\pp^{n-1}}
    \end{align*}
    By fixing the denominator we see that up to scaling we must have \begin{align}
    \label{exprseeionforf'}
        f'= \left.\frac{f}{\prod_{G\nmid H}l_G}\right|_H
    \end{align}
    
    First we show that $f'$ actually is a polynomial, i.e. that $f|_H$ is divisible by $\prod_{G\nmid H}l_G|_H$. As already stated above, no three facets of a convex polytope intersect in a codimension 2 space and so for $G_1\neq G_2$ we have $G_1\cap H\neq G_2 \cap H$ and therefore the two linear forms $l_{G_1}|_H,l_{G_2}|_H$ don't have a common divisor. Hence it suffices to show that $f|_H$ is divisible by $l_G|_H$ for any facet $G\nmid H$. Fix such a facet $G\nmid H$. Within $G\cap H$ there is a Zariski dense open subset $U$ such that for $x\in U$ we have $\cF_x=\{G,H\}$. More explicitly $U$ avoids all of the spaces $H\cap G \cap G'$ for $G'\in \cH_P\setminus \{H,G\}$. Since such spaces are of codimension at least 3 in $\pp^n$, the set $U$ is in fact Zariski dense. Now it suffices to show that for any $x \in U$ we have $f(x)=0$, since then by taking the Zariski closure we get $f(x)=0$ for all $x \in G\cap H$ and hence by Hilbert's Nullstellensatz $l_G|H$ divides $f$. Pick $x\in U$, then \begin{align*}
        \mu_x(f)\geq \ord_P(L_x)=\ord_P(G\cap H)=1.
    \end{align*}
    The final identity is because $\nl_{M_P}(\cF_x)=0$ and $G\cap H$ does not intersect $P$ in a codimension 2 face since $G\nmid H$. But by definition $\mu_x(f)\geq 1$ means $f(x)=0$ and hence we conclude that indeed $f'$ is a polynomial.

    From (\ref{exprseeionforf'}) we also see that \begin{align*}
        \deg(f')&=\deg(f)-|\{G\in \cH_P| G\nmid H\}|\\&= d-n-1-(d-d'-1)=d'-(n-1)-1
    \end{align*}
as claimed.

    Next we need to get the bounds on the orders of vanishing for $f'$. We start by some general observations regarding the relationship of the polytope $P$ and the polytope $P'=P\cap H$. \begin{enumerate}[label=(\alph*)]
        \item As noted above we may view $\cH_{P\cap H}$ as a subset of $\cH_P$. Then the matroid $M_{P'}$ is a minor of the matroid $M_P$. More precisely $M_{P'}$ is obtained from $M_P$ by contracting the element $H$ and restricting to $\cH_{P'}$. In particular for $S\subseteq \cH_{P'}$ we get $\nl_{M_{P'}}(S)=\nl_{M_P}(S\cup\{H\})$.
        \item Fix a point $x \in H$. The flat $\cF_x'\subseteq \cH_{P\cap H}$ of all hyperplanes passing through $x$ is obtained by removing $H$ and all facet hyperplanes not incident with $H$ from $\cF_x$. In particular $L_x \subseteq L_x'$ in $\pp^n$.
    \end{enumerate}
        Let $G_1,\ldots,G_s$ be an enumeration of $\cF_x\setminus (\cF_x'\cup \{H\})$ such that there exists $0\leq j\leq s$ with \begin{align*}
            L_x'&\subseteq G_1,\ldots, G_j \\
             L_x'&\nsubseteq G_{j+1},\ldots, G_s 
        \end{align*}
        Notice that this implies \begin{align*}
            \rk_{M_{P'}}(\cF_x')\overset{(a)}{=}\rk_{M_P}(\cF_x'\cup\{H\})-1= \rk_{M_P}(\cF_x'\cup\{H,G_1,\ldots,G_j\})
        \end{align*}
        and therefore \begin{align}
        \label{eq:nullity_ineq}
            \nl_{M_P}(\cF_x'\cup\{H,G_1,\ldots,G_j\})= \nl_{M_P}(\cF_x'\cup\{H\})+j= \nl_{M_{P'}}(\cF_x')+j
        \end{align}
        Let us consider the linear subspace $L_x'$. Any hyperplane in $\cH_P$ other than those belonging to $\{H,G_1,\ldots,G_j\}\cup \cF_x'$ intersects $L_x'$ in a subspace of codimension at least one in $L_x'$. Let $U\subseteq L_x'$ denote the Zariski dense subset in $L_x'$ which avoids all hyperplanes in $\cH_P$ which do not appear in the list $\{H,G_1,\ldots,G_j\}$. Fix an arbitrary point $y\in U$. We must have $\cF_y=\cF_x'\cup\{H,G_1,\ldots,G_j\}$ and therefore $L_y=\bigcap \cF_y=L_x'$ by definition of $G_1,\ldots,G_j$. In particular by 
        (\ref{eq:nullity_ineq}) we obtain \begin{align}
               \label{eq:order_ineq_y}
               \ord_P(L_y)=\ord_{P'}(L_x')+j 
        \end{align}
 Since for $G\nmid H$ we have $\mu_y(l_G)=1$ if $G\in \{G_1,\ldots,G_j\}$ and $\mu_y(l_G)=0$ otherwise, we conclude \begin{align*}
     \mu_y(f')=\mu_y(f)-\sum_{G\nmid H}\mu_y(l_G)= \mu_y(f)-j\geq \ord_{P}(L_y)-j= \ord_{P'}(L_x').
 \end{align*}
  As $y\in U$ was arbitrary and $U$ is Zariski dense in $L_x'$ we conclude that $f'$ vanishes to at least the order $\ord_{P'}(L_x')$ at any point in $L_x'$. In particular at the point $x$ we see \begin{align*}
     \mu_x(f')\geq \ord_{P'}(L_x')
 \end{align*}
 which finishes the proof.
\end{proof}

Our next step is to get rid of the assumption that $f$ does not identically vanish along $H$ in Lemma \ref{crucial_Lemma}. To this end we will show that if $f$ satisfies the vanishing conditions of Theorem \ref{thm:intro_main_allcomponents} then $f$ can not vanish identically along any facet hyperplane of $P$. The idea here is that the fixed degree of $f$ prevents $f$ from "vanishing too much in too many places at once".
We establish this fact by first finding one facet along which $f$ does not vanish identically and then use Lemma \ref{crucial_Lemma} for this facet to reduce dimension.
\begin{lem}
    Assume $f\in \C[X_0,\ldots,X_n]_{d-n-1}$ is homogeneous of degree $d-n-1$ and $f$ vanishes at each point $x\in \pp^n$ to order at least $\ord_P(x)$. Then for any hyperplane $H\in \cH_P$ we have $f|_H\neq 0$, i.e. $f$ does not identically vanish along $H$.
\end{lem}
\begin{proof}
    First notice that if $f$ vanishes identically along a hyperplane $H$ then $l_H$ divides $f$. Hence if $f$ vanished along all hyperplanes $H\in \cH_P$ then as the equations of pairwise distinct linear hyperplanes share no common factors, we would have $\deg(f)\geq d$. But clearly $\deg(f)=d-n-1 < d$, hence we find at least one hyperplane $H$ along which $f$ does not identically vanish. \par
    Assume for contradiction that there exists at least one hyperplane $G\in \cH_P$ along which $f$ is identically zero. Then we can choose $G,H\in \cH_P$ such that $f$ is not zero on all of $H$, $f$ is zero along all of $G$ and $G,H$ correspond to incident facets of $P$. This is always possible because the incidence graph of the facets of a polytope is connected (it is the 1-skeleton of $P^\circ$), so we can walk from any ``bad'' hyperplane $G$ to any ``good'' hyperplane $H$ and need to cross from bad to good somewhere. By Lemma \ref{crucial_Lemma} we see that the polynomial \begin{align*}
        f':=\left.\frac{f}{\prod_{K\nmid H}l_K}\right|_H
    \end{align*} 
    has the correct degree $\deg(f')=|\cH_{P'}|-\dim(P')-1$ and satisfies $\mu_x(f') \geq \ord_{P'}(x)$ for the polytope $P'=P\cap H$. Furthermore since $G$ was incident with $H$, the equation $l_G$ does not appear in the denominator here, so $f'$ also identically vanishes along the facet $G\cap H$ of $P'$. Recursively we reduce to the case $n=1$. Here $P$ is a line segment in $\R^1$ which has two facets, so the degree of $f$ is $2-1-1=0$ and $f$ is forced to be a non-zero constant. In particular $f$ can not vanish at any of the two facets of $P$.
\end{proof}

Finally we can now prove uniqueness, since by the previous lemma we know that Lemma \ref{crucial_Lemma} is applicable for any hyperplane $H\in \cH_P$. The only difficulty that is left is the fact that the vanishing order can of course at most fix a polynomial up to scaling. However, to utilize the theory of canonical forms for positive geometries we need to find one scaling such that simultaneously the residue at all vertices is $\pm 1$. We carry this out in the following proposition.

\begin{prop}
\label{dealing with scalings}
    Assume $f\in \C[X_0,\ldots,X_n]_{d-n-1}$ is homogeneous of degree $d-n-1$ and vanishes at each point $x\in \pp^n$ to order at least $\ord_P(L_x)$. Then up to scaling \begin{align*}
        \Omega_f(P)=\Omega(P).
    \end{align*}
\end{prop}
\begin{proof}
    By Definition \ref{defi_positive_geom} and since polytopes are positive geometries with unique canonical form $\Omega(P)$, it suffices to show that there exists a scalar $c\in \C$ such that \begin{align*}
        \Res_H(c\Omega_f(P))=\Res_H\Omega(P).
    \end{align*} 
    holds for all $H\in \cH_P$.
    The case $n=0$ of a single point is trivial since $f$ is of degree zero here and can therefore be scaled to $\pm 1$.

    For the general case consider a facet hyperplane $H\in \cH_P$. Using Lemma \ref{crucial_Lemma} and by induction we find a scalar $c_H$ depending on $H$ such that 
    \begin{align*}
        \Res_H(c_H\Omega_f(P))=\Omega(P\cap H)=\Res_H\Omega(P).
    \end{align*}
    It remains to prove that this scalar $c_H$ can be chosen independently of $H$. In other words we need to show $c_H=c_G$ for any two facet hyperplanes $G,H\in \cH_P$. Since the incidence graph of a convex polytope is connected, we may further assume that $G,H$ are two incident facets. 
    Taking residue twice is skew-symmetric due to the different choice of orientation (see e.g. \cite[p.~27]{CattaniDickenstein2005Residues}). Hence taking residue twice in two different orders we get \begin{align*}
        c_G\Res_H\Res_G(\Omega_f(P))&=\Res_H\Res_G(c_G\Omega_f(P))\\&=\Res_H\Res_G(\Omega(P))\\&=-\Res_G\Res_H(\Omega(P))\\&= -\Res_G\Res_H(c_H\Omega_f(P))\\
        &=-c_H\Res_G\Res_H(\Omega_f(P))\\
        &=c_H\Res_H\Res_G(\Omega_f(P))
    \end{align*}
    which implies $c_H=c_G$.
\end{proof}
Let us conclude by deducing our main result Theorem \ref{thm:intro_main_allcomponents} from this.
 \begin{proof}[Proof of Theorem \ref{thm:intro_main_allcomponents}]
     The existence and the fact that Warrens adjoint of the dual polytope $\adj_P$ satisfy these restrictions was proved in Proposition \ref{prop:existence}. If $f,f'$ are two polynomials of the correct degree that satisfy the vanishing restrictions then by Proposition \ref{dealing with scalings} we find two scalars $c,c'$, such that \begin{align*}
         c\frac{f}{\prod_{G\in \cH_P}l_G}\omega_{\pp^n}= \Omega(P)= c'\frac{f'}{\prod_{G\in \cH_P}l_G}\omega_{\pp^n}
     \end{align*}
     which by clearing the denominators clearly implies $cf=c'f'$, i.e. $f$ and $f'$ agree up to scaling.
 \end{proof}
 
\section{Reduction to Interpolation in Dimension Zero } \label{sec: main_cor_Clemens}
In this section we will deduce Theorem \ref{thm:intro_main_thm} from Theorem \ref{thm:intro_main_allcomponents}. Throughout $P$ will be a fulldimensional convex polytope in $\pp^n$. Recall the following notation.
\begin{defin}
 Let $P \subset \pp^n$ be a convex polytope. We call the set
 \begin{equation*}
     \cR_0(P) := \left\lbrace x \mid x = \bigcap\cF_x \right\rbrace
 \end{equation*}
 the \emph{point-residual} of $P$. For $x\in \cR_0(P)$ we also write $\ord_P(x):=\ord_P(\{x\})$ for the order defined on all flats in Definition \ref{def:order_general}.
\end{defin}
As before we write $L_x = \bigcap\cF_x$, so the point-residual is defined by the condition $x = L_x$. Yet differently $\cR_0(P)$ is the union of the zero-dimensional part of the residual arrangement $\cR(P)$ of Kohn and Ranestad \cite{KohnRanestad2019AdjCurves} and the vertices of $P$. In our notation it is also the subset of $\cL(P)$ of spaces of dimension zero. The order at a point $x \in \cR_0(P)$ can be written in a slightly simpler form since compared to Definition \ref{def:order_general} we now only consider rank $n$ flats: \begin{align*}
    \ord_P(x)=\begin{cases}
        |\{H\in \cH_P \mid x \in H\}|-n & \text{if }x\text{ is a vertex of }P \\
        |\{H\in \cH_P \mid x \in H\}|-n+1 &\text{else}
    \end{cases}
\end{align*}

The difference between Theorem \ref{thm:intro_main_allcomponents} and
Theorem \ref{thm:intro_main_thm} only lies in the vanishing conditions for the adjoint that we postulate: Theorem \ref{thm:intro_main_allcomponents} gives vanishing conditions at flats of any rank, while Theorem \ref{thm:intro_main_thm} only considers flats of rank $n$, i.e. finitely many points in $\pp^n$. Hence the existence part is stronger in Theorem \ref{thm:intro_main_allcomponents}, but Theorem \ref{thm:intro_main_thm} significantly strengthens the uniqueness part. In particular we can immediately deduce the existence part of Theorem \ref{thm:intro_main_thm}.
\begin{cor}
\label{cor:existence}
    The adjoint polynomial $\adj_P$ satisfies \begin{align*}
        \mu_x(\adj_P)\geq \ord_P(x)
    \end{align*} for all $x\in \cR_0(P)$. Hence the existence statement in Theorem \ref{thm:intro_main_thm} holds true. 
\end{cor}
\begin{proof}
    Apply Proposition \ref{prop:existence} only to the points $x \in \cR_0(P)$.
\end{proof}

We will subsequently show that the hypotheses in \Cref{thm:intro_main_allcomponents} are stronger than necessary for the uniqueness statement. Indeed, we show that it suffices to postulate the vanishing order in the finitely many points $x \in \cR_0(P)$ to deduce the correct vanishing order along any flat: 

\begin{thm} \label{thm:pointresidual}
 Let $P \subset \pp^n$ be a full-dimensional convex polytope with $d$ facets and let $f \in \C[X_0, \dots, X_n]_{d-n-1}$ be a polynomial satisfying
 \begin{equation*}
     \ord_x(f) \geq \ord_P(x)
 \end{equation*}
 for all $x \in \cR_0(P)$. Then
 \begin{equation*}
     \ord_x(f) \geq \ord_P(L_x)
 \end{equation*}
 for all $x \in \pp^n$.
\end{thm}

From this together with Theorem \ref{thm:intro_main_allcomponents}, the uniqueness part of Theorem \ref{thm:intro_main_thm} follows immediately. Together with Corollary  \ref{cor:existence} this concludes the proof of Theorem \ref{thm:intro_main_thm}.

We will now prove \Cref{thm:pointresidual}. The main idea is that the low degree of $f$ means that if $f$ vanishes to a high enough order at a large enough number of points in some linear subspace, then $f$ must vanish to a high order along the whole subspace. We begin with some preparatory lemmata.

\begin{lem} \label{lem: linear_van_order_Mult}
Fix $D,k\in \N$.
 Let  $f \in \R[x_0, \dots, x_n]_D$ be homogeneous of degree $D$ and let $H_1, \dots, H_{k}$ be pairwise distinct hyperplanes in $\pp^{n}$ such that
 \begin{equation*}
     \sum_{i=1}^k \mu_{H_i}(f) > D.
 \end{equation*}
  Then $f = 0$.
\end{lem}

\begin{proof}
For $i=1,\ldots,k$ let $l_i$ be a linear equation cutting out $H_i$. As the $H_i$ are distinct, the $l_i$ don't share any factors, so by definition of multiplicity
\begin{align*}
    f\in \bigcap_{i=1}^k \langle l_i\rangle^{\mu_{H_i}(f)}=
    \bigcap_{i=1}^k \langle l_i^{\mu_{H_i}(f)}\rangle
    =\langle \prod_{i=1}^kl_i^{\mu_{H_i}(f)}\rangle
\end{align*}  
where in the last equality we use that the intersection of principle ideals equals the ideal generated by the least common multiple of their generators. The ideal on the right is principal with a generator of degree $\sum_{i=1}^k\mu_{H_i}>D=\deg(f)$, hence $f=0$.

\end{proof}
Applying Lemma \ref{lem: linear_van_order_Mult} inductively to restrictions of $f$, we also get a criterion for $f$ vanishing to a certain order along a linear subspace. We carry this out in the following corollary which will be our main tool for the proof of Theorem \ref{thm:pointresidual}.
\begin{cor} \label{cor: van_order_Mult}
 Fix $D,k\in \N$, let $L\subset \pp^{n}$ be a hyperplane and let $f \in \C[x_0, \dots, x_{n}]_D$ be homogeneous of degree $D$. Furthermore let $H_1, \dots, H_k$ be pairwise distinct hyperplanes in $L$. Assume that there exists $m \in \N$ with
 \begin{equation*}
     \sum_{i=1}^k (\mu_{H_i}(f)-(m-1)) > D - (m -1).
 \end{equation*}
 Then $\mu_L(f) \geq m$.
\end{cor}

\begin{proof}
 We use induction on $m$. If $m = 1$, then we simply apply \Cref{lem: linear_van_order_Mult} to $f|_L$. For $m \geq 2$ we have
 \begin{equation*}
     \sum_{i=1}^k (\mu_{H_i}(f)-(m-2)) \geq \sum_{i=1}^k (\mu_{H_i}(f)-(m-1)) +1 > D - (m - 1) + 1 = D-(m-2),
 \end{equation*}
 and so $\mu_L(f) \geq m-1$ by the inductive hypothesis. In particular, fixing a linear equation $l_L$ of $L$, we see that $f' := \frac{f}{l_L^{m-1}}$ is a polynomial of degree $D'=D-(m-1)$. By assumption
 \begin{equation*}
     \sum_i \mu_{H_i}(f') = \sum_i (\mu_{H_i}(f)-(m-1)) > D - (m - 1)=D',
 \end{equation*}
 so $f'|_L = 0$ by \Cref{lem: linear_van_order_Mult} applied to $f'|_L$. This implies $\mu_L(f)=\mu_L(f')+m-1\geq 1+m-1=m$.
\end{proof}

Finally we need two small statements about convex polytopes.
\begin{lem}
\label{lem: convpol_noninc}
Let $P\subseteq \pp^n$ be a fulldimensional convex polytope.
 Consider a linear space $L \subseteq \pp^n$ of dimension $l$ such that $\dim(L\cap P)< \dim(L)$. Then $L\cap P$ contains at most one face of $P$ of dimension $l-1$.
\end{lem}

\begin{proof}
Assume the contrary. Then $L\cap P$ contains two distinct $(l-1)$-faces of $P$. Necessarily it will also contain their convex hull, which has dimension at least $l$. In particular, the intersection $L \cap P$ has dimension at least $l$, which is a contradiction to our assumption.
\end{proof}

We now have all the tools for the main proof of this section.

\begin{proof}[Proof of \Cref{thm:pointresidual}]
Let $f \in \C[X_0, \dots, X_n]_{d-n-1}$ vanish in the point-residual, i.e.
\begin{equation*} \label{eq:pointresidual}
    \mu_x(f) \geq \ord_P(x)
\end{equation*}
for all $x \in \cR_0(P)$. We want to prove that \begin{align*}
    \mu_{L_x}(f)\geq \ord_P(L_x)
\end{align*}must already hold for all $x \in \pp^n$. We will do so by induction on
\begin{equation*}
    r = \dim L_x = n-\rk(\cF_x).
\end{equation*}
The case $r = 0$ describes the vanishing in the point-residual, which is covered by our hypothesis. Thus let $r > 1$. Fix an enumeration $(\cF_i)_{i=1}^k$ of all flats of $M_P$ satisfying $\cF_x \subseteq \cF_i$ and $\rk(\cF_x) = \rk(\cF_i) - 1$. Geometrically, $L_i:=\bigcap\cF_i$ is a hyperplane in $L_x$.

We want to invoke \Cref{cor: van_order_Mult} to prove that $f$ vanishes in $x$ to the right order. As suggested by the definition of $\ord_P(L_x)$, we distinguish between two cases depending on whether the inequality $\dim(L_x\cap P)\leq \dim (L_x)$ is strict.
\begin{enumerate}
    \item Assume that $\dim(L_x\cap P)= \dim (L_x)$ and hence $\ord_P(L_x)=\nl_{M_P}(\cF_x)$. We wish to apply Corollary \ref{cor: van_order_Mult} to $D=d-n-1,m=\nl_{M_P}(\cF_x)$ and $H_i=L_i, i=1,\ldots,k$. We obtain that $\mu_{L_x}(f)\geq \nl_{M_P}(\cF_x)=\ord_P(L_x)$ is implied by 
    \begin{equation*}
    \begin{split}
        \sum_{i=1}^k \left(\mu_{L_i}(f) - (\nl_{M_P}(\cF_x) - 1)\right) &\geq d-n-1 - (\nl_{M_P}(\cF_x) - 1)
    \end{split}
    \end{equation*}
    To see that this is satisfied, we first note that by our inductive hypothesis and by definition of $\ord_P(L_i)$ we have
    \begin{equation*}
        \mu_{L_i}(f) \geq \ord_P(L_i) \geq \nl_{M_P}(\cF_i) 
    \end{equation*}
    for all $i$. Therefore
    \begin{equation*}
    \begin{split}
        \sum_i \left( \mu_{L_i}(f) - (\nl_{M_P}(\cF_x) - 1) \right) &\geq \sum_i \left( \nl_{M_P}(\cF_i) - \nl_{M_P}(\cF) + 1 \right) = \\
        &= \sum_i \left(|\cF_i| - |\cF_x| + \rk(\cF_x) - \rk(\cF_i) + 1\right) = \\
        &= \sum_i |\cF_i \setminus \cF_x| = d  - |\cF_x| > \\
        &> d - n - 1 - (\nl_{M_P}(\cF_x) - 1).
    \end{split}
    \end{equation*}
    The first equality in the penultimate line comes from the fact that the sets $\cF_i\setminus \cF_x$ form a partition of $\cH_P\setminus \cF_x$. For the last inequality we use that since $r > 0$ we have $n > \rk(\cF_x) = |\cF_x| - \nl_{M_P}(\cF_x)$.
    
    \item Now assume that $\dim(L_x\cap P)< \dim (L_x)$ and therefore $\ord_P(L_x)=\nl_{M_P}(\cF_x)+1$. This time we apply Corollary \ref{cor: van_order_Mult} to $D=d-n-1,m=\nl_{M_P}(\cF_x)+1$ and $H_i=L_i, i=1,\ldots,k$. We obtain that $\mu_{L_x}(f)\geq \nl_{M_P}+1(\cF_x)=\ord_P(L_x)$ is implied by 
    \begin{equation*}
    \begin{split}
        \sum_{i=1}^k \left(\mu_{L_i}(f) - \nl_{M_P}(\cF_x)\right) &\geq d-n-1 - \nl_{M_P}(\cF_x)
    \end{split}
    \end{equation*}

    As in case (1) we have
    \begin{equation*}
        \mu_{L_i}(f) \geq \ord_P(L_i) \geq \nl_{M_P}(\cF_i) 
    \end{equation*}
    for all $i$. The second inequality by definition is an equality precisely when $\dim(L_i\cap P)=\dim(L_i)$. By \Cref{lem: convpol_noninc} this can happen for at most one index $i\in \{1,\ldots,k\}$.  For the remaining indices, the inequality is strict. Therefore we conclude as in the first case:
    \begin{equation*}
    \begin{split}
        \sum_i \left( \mu_{L_i}(f) - \nl_{M_P}(\cF_x) \right) \geq &\sum_i \left( \nl_{M_P}(\cF_i) + 1 - \nl_{M_P}(\cF_x) \right) - 1 = \\
        &= \sum_i \left(|\cF_i| - |\cF_x|\right) - 1 = \sum_i \left( |\cF_i \setminus \cF_x| \right) - 1 = \\
        &= d - |\cF_x| - 1 \\&> d - n - 1 - \nl_{M_P}(\cF_x).
    \end{split} 
    \end{equation*}
\end{enumerate}
\end{proof}

\section{Examples and Counterexamples} \label{sec: exmpls}
In this section we collect some interesting examples. We also provide counterexamples to many natural generalizations and strengthenings of Theorem \ref{thm:intro_main_thm} and Theorem \ref{thm:intro_main_allcomponents} that one might hope for. We start by showcasing how $\cR_0(P)$ can be computed in {\tt Macaulay2} \cite{M2}.

Let $P\subseteq \pp^n$ be a full-dimensional polytope with $d$ facets, then after fixing coordinates on $\pp^n$ we can record the homogeneous equations of the $d$ facet hyperplanes of $P$ in an $(n+1)\times d$-matrix $F$. For the purpose of using a computer algebra system we assume that the entries of $F$ are rational.
Using the {\tt Macaulay2} \cite{M2} package "Matroids" (\cite{M2PackageMatroids, MatroidsSource}),  we can compute the point residual $\cR_0(P)$ using the "hyperplanes" command:

\beginOutput
i1 : needsPackage "Matroids";\\
i2 : M = matroid F\\
i3 : L = hyperplanes M\\
i4 : R\_0 = transpose matrix apply(L, j-> flatten entries gens \\ \hspace{1cm} kernel((transpose submatrix(F,toList j))))  
\endOutput

The output is a matrix whose columns are the coordinates of the points in $\cR_0(P)$.

\begin{ex}
    Let $P$ be the three-dimensional octahedron, i.e. \begin{align*}
        P=\conv\left(\begin{pmatrix}
            1 \\ 0 \\ 0
        \end{pmatrix},\begin{pmatrix}
            -1 \\ 0 \\ 0
        \end{pmatrix},\begin{pmatrix}
            0 \\ 1 \\ 0
        \end{pmatrix},\begin{pmatrix}
            0 \\ -1 \\ 0
        \end{pmatrix},\begin{pmatrix}
            0 \\ 0 \\ 1
        \end{pmatrix},\begin{pmatrix}
            0 \\ 0 \\ -1
        \end{pmatrix}\right)\subseteq \R^3\subseteq \pp^3
    \end{align*}
The homogeneous equations of the $8$ facets of $P$ are (up to scaling) the columns of the matrix \begin{align*}
    F=\left(\!\begin{array}{cccccccc}
      -1&1&-1&1&-1&1&-1&1\\
      -1&-1&1&1&-1&-1&1&1\\
      -1&-1&-1&-1&1&1&1&1\\
      -1&-1&-1&-1&-1&-1&-1&-1
      \end{array}\!\right)
\end{align*}
After running the commands above for this $F$ we obtain a $4\times 20$ matrix ${\tt R\_0}$.
 The columns of the matrix ${\tt R\_0}$ are the projective coordinates of all 20 points in the point residual $\cR_0$: \begin{align*}
\hspace{-1cm}
\left(\!\begin{array}{cccccccccccccccccccc}
       0&0&0&1&-1&-1&-1&1&1&1&1&-1&1&-1&-1&0&0&1&-1&0\\
       0&1&-1&0&0&1&1&-1&1&1&1&-1&0&0&1&1&-1&-1&-1&0\\
       1&0&1&0&1&0&1&1&-1&1&0&1&1&0&-1&1&0&-1&-1&-1\\
       1&1&0&1&0&0&1&1&1&1&0&1&0&1&1&0&1&1&1&1
       \end{array}\!\right) 
\end{align*} 
The orders $\ord_P(x)$ of these 20 points are given by \begin{align*}
    (1, 2, 1, 1, 1, 1, 2, 1, 1, 1, 2, 2, 2, 1, 1, 2, 1, 1, 1, 1).
\end{align*}
Now we can set up a linear system to solve for the coefficients of a homogeneous degree $8-3-1=4$ polynomial in $4$ variables which vanishes at the 20 columns of the matrix ${\tt R\_0}$ and whose first partial derivatives vanish at the 6 special columns with indices 2, 7, 11, 12, 13, 16, where the order is $\ord_P(x)=2$. Up to scaling we find that the unique solution to this linear system is given by the coefficients of the adjoint\begin{align*}
    \adj_P(X)&=-X_0^4+2X_0^2X_1^2+2X_0^2X_2^2-2X_0^2X_3^2-X_1^4+2X_1^2X_2^2-2X_1^2X_3^2-X_2^4 \\&-2X_2^2X_3^2+3X_3^4.
\end{align*}
\end{ex}

The adjoint of a pyramid is a cone with apex the tip of the pyramid over the adjoint of the base of the pyramid. This process usually introduces a singularity of the adjoint and we can read this off from the  order at the tip. The next example carries this out for a pyramid over a regular pentagon.

\begin{ex}
    \begin{enumerate}
        \item Let $P_1\subseteq \R^2$ be a regular pentagon. Then $\cR_0(P_1)$ consists of the five vertices $v_1,\dots,v_5$ with order $\ord_{P_1}(v_i)=0$ and the five points $x_1,\ldots,x_5$ obtained by intersecting two non-adjacent facets with order $\ord_{P_1}(x_i)=1$. Since $P_1$ is regular, the points $x_1,\ldots,x_5$ lie on a perfect circle which is the adjoint hypersurface $A_{P_1}=\cV(\adj_{P_1}(X))$, see the left image in Figure \ref{fig:penta_adj}.
    \item Now let $P_2$ be the pyramid over the regular pentagon $P_1$ in $\pp^3$. We have $\cR_0(P_2)=\cR_0(P_1)\cup\{v\}$, where $v$ is the tip of the pyramid $P_2$. The order at the tip is $\ord_{P_2}(v)=5-3=2$ since five facets pass through this vertex. We immediately know that the adjoint hypersurface $A_{P_2}$ must have at least one singular point. Clearly the cone through the tip $v$ over the adjoint line of the pentagon $P_1$ satisfies the correct vanishing conditions at all points in $\cR_0(P_2)$, so by Theorem \ref{thm:intro_main_thm} this must be the adjoint hypersurface of $P_2$, see the right image of Figure \ref{fig:penta_adj}.
    \begin{figure}
		\includegraphics[height=3cm]{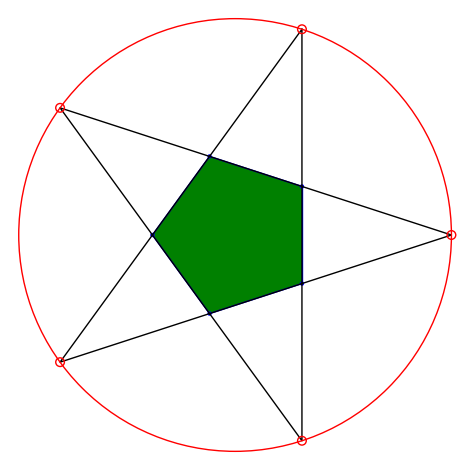}
        \includegraphics[height=3cm]{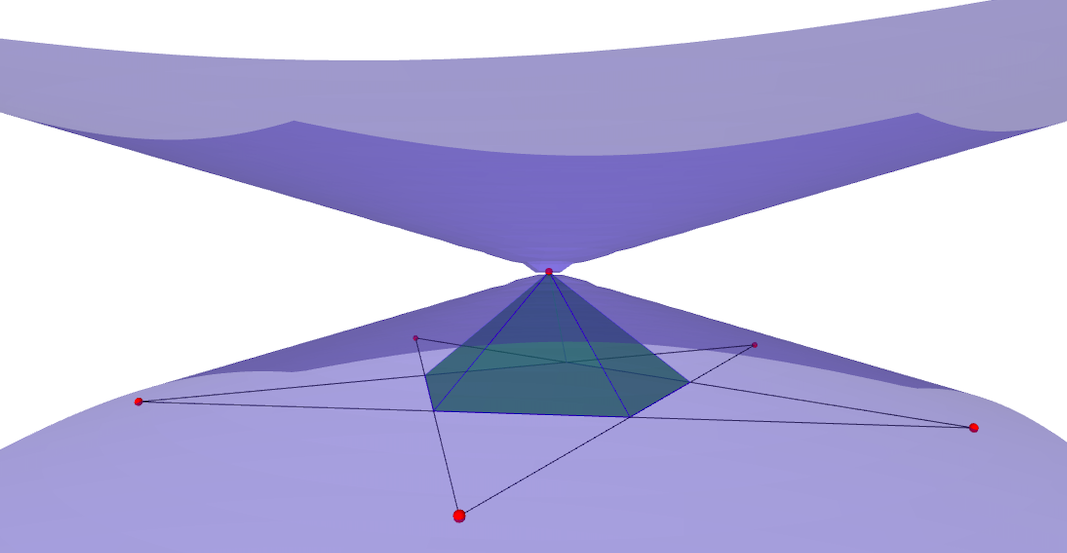}
        \caption{Left: The adjoint of the regular pentagon $P_1$ is a perfect circle through 5 points. Right: The adjoint of the pyramid $P_2$ is the cone over the adjoint of the pentagon $P_1$ through the tip of the pyramid $P_2$. The tip is a singularity.}
        \label{fig:penta_adj}
	\end{figure}
    \item Finally we may cut off the tip of the pyramid $P_2$ introducing a new facet that is parallel to its base. This yields the polytope $P_3$ depicted in the top of Figure \ref{fig:stump}. We have $\cR_0(P_3)=\cR_0(P_2)\cup\{y_1,\ldots,y_5\}$ where the $y_i$ are the five points at infinity obtained by intersecting the two parallel facets of $P_3$ with any of the remaining 5 facets. We have $\ord_{P_3}(y_i)=1$ and for the tip $v$ of $P_2$ we now have $\ord_{P_3}(v)=\ord_{P_2}(v)+1=3$, since $v$ no longer is a face of $P_3$. Compared to (2) we need one more order of vanishing at $v$ and we need to vanish at $y_1,\ldots,y_5$. Satisfying these additional conditions is possible, since the adjoint now has degree 3. Since all the points $v,y_1,\ldots,y_5$ lie on a hyperplane, these conditions can be satisfied by multiplying the adjoint of $P_2$ with the linear form cutting out this plane. Hence the adjoint hypersurface $A_{P_3}$ is reducible with two components: $A_{P_2}$ and a hyperplane (see the bottom image in Figure \ref{fig:stump}).
    \begin{figure}
        \includegraphics[height=3cm]{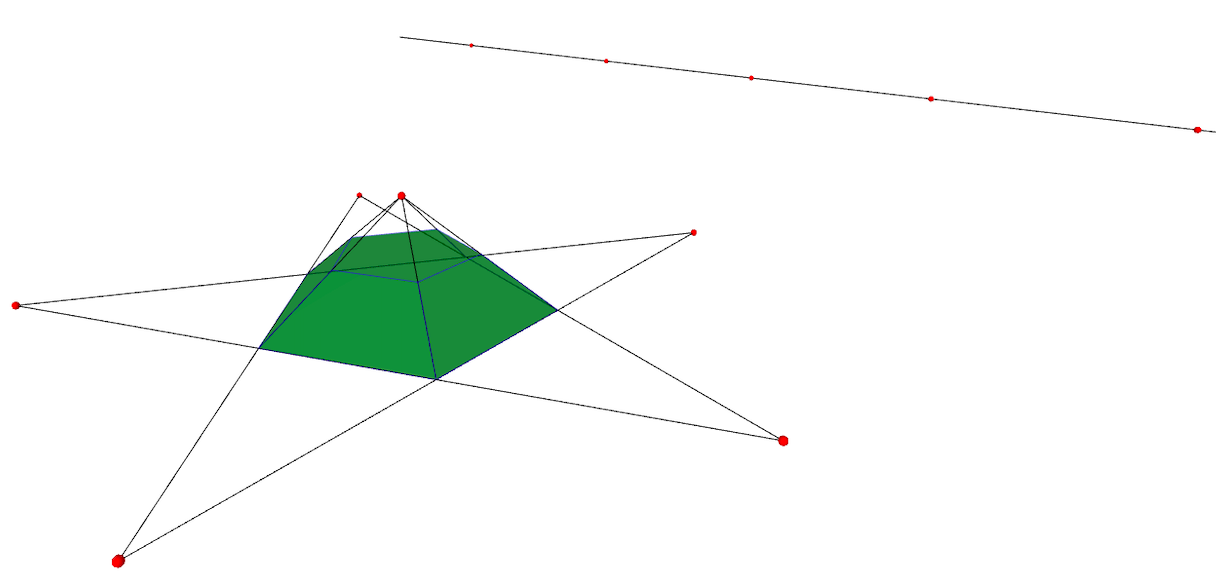}
        \includegraphics[height=3cm]{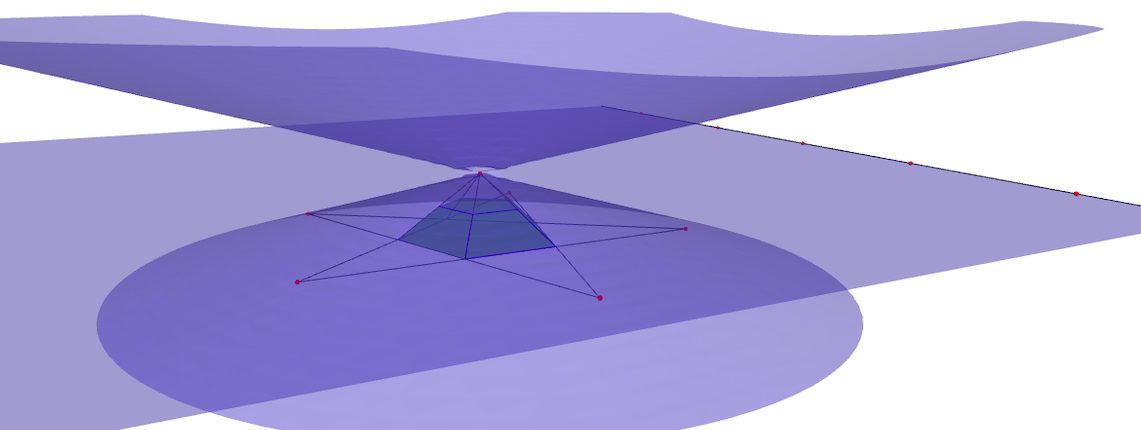}
        \caption{Top: The polytope $P_3$ and its residual points of order $>0$. The line with 5 points depicted in the back is at infinity. Bottom: The adjoint of $P_3$ factors through the adjoint of $P_2$ and a hyperplane}
        \label{fig:stump}
	\end{figure}
    \end{enumerate}
\end{ex}

The next examples showcase that the order of vanishing which we require at residual points is in general only a lower bound. Furthermore the adjoint can also have singular points entirely outside the union of linear spaces $L\in \cL(P)$ that are intersections of (at least one) facet hyperplanes of $P$.

\begin{ex}
 Consider a polytope $P$ with simple hyperplane arrangement and combinatorial type as given in \cite[Table 1, row 7]{KohnRanestad2019AdjCurves} (see \Cref{fig:unpredictedSglty} for a sketch of the combinatorial type and the residual arrangement). Its adjoint hypersurface $A$ is a cubic surface with a singularity (at least) in the unique point $x$ where three residual lines meet. This follows because the three lines span a three-dimensional tangent space to $A$. Since $\dim(A) = 2$, this means that $A$ cannot be smooth in this point. At the same time $x \in \cR_0(P)$ and its vanishing order is $\ord_P(x) = 1$, since $x$ is not a vertex and $P$ is simple.
 \begin{figure}[h]
     \centering
     \includegraphics[width=0.6\linewidth, trim={0 0 8cm 0}, clip]{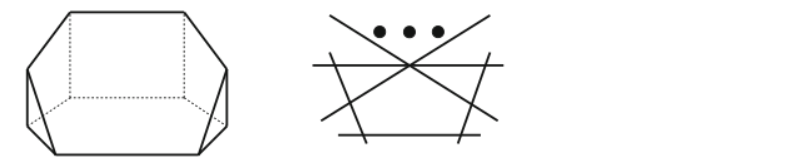}
     \caption{A combinatorial type of a three-dimensional polytope with simple hyperplane arrangement (left). Its residual arrangement indicates a singularity at the point where three lines meet (images taken from \cite[Table 1]{KohnRanestad2019AdjCurves}).}
     \label{fig:unpredictedSglty}
 \end{figure}
\end{ex}

\begin{ex}
 Consider a regular hexagon in the plane. Its adjoint is a union of a circle and the line at infinity. The two (complex-valued) points where these two components meet are singular, which, however, is not seen in the order of these points, since the hexagon has simple hyperplane arrangement.
\end{ex}

\begin{ex}
 Consider the regular cube $P \subset \R^3 \subset \pp^3$. Its residual arrangement comprises three lines $L_1, L_2, L_3$ contained in the plane at infinity. Following \Cref{def:order_general} they all have vanishing order $\ord_P(L_i) = \nl_{M_P}(L_i) + 1 = 1$.

 The point residual $\cR_0(P)$ of $P$ comprises the six vertices with vanishing order 0 (since $P$ is simple as a polytope) and three points in the plane at infinity. Each of the latter has vanishing order $\ord_P(x) = 2$, since it is contained in four facet hyperplanes. From this we conclude that the adjoint hypersurface of $P$ equals the plane at infinity, counted with multiplicity 2. In particular, each point on the adjoint is singular.
\end{ex}

\begin{ex}
 Consider a polytope $P$ with simple hyperplane arrangement in $\R^n$ where $4 \leq n$ and assume that its adjoint hypersurface has degree two. It is shown in \cite[Proposition 3.3]{BrueserKummerPavlov25Adjoints} that if $\cR(P)$ contains a linear subspace of codimension two in $\pp^n$, then the adjoint hypersurface has a determinantal representation and thus is singular by \cite[\S 4.1.1]{Dolgachev2012CAG}. However, by simplicity of $P$, these singularities can not be predicted using our results, since we have $\ord_P(x) \leq 1$ for all $x \in \pp^n$.

 For a concrete example consider the 4-polytope below, given by its defining inequalities, which is obtained by truncating two vertices of a scaled version of the standard 4-simplex:
 \begin{equation*}
 \begin{split}
     x_1 \geq 0, \quad
     x_2 \geq 0, \quad
     x_3 \geq 0, \quad
     x_4 \geq 0 \\
     x_1 + x_2 + x_3 + x_4 \leq 4 \\
     2x_1 + x_3 \leq 6, \quad 2x_2 - x_4 \leq 6
 \end{split}
 \end{equation*}
 The facets obtained by the truncation share no vertices, hence their intersection is residual and has codimension 2. The adjoint polynomial equals
 \begin{equation*}
     \alpha = 72x_0^2-18x_0x_1-18x_0x_2+4x_1x_2-12x_0x_3+3x_2x_3+12x_0x_4-3x_1x_4-2x_3x_4,
 \end{equation*}
 which is singular in the point $(1:0:0:6:-6)$.
\end{ex}

\bibliographystyle{alpha}
\bibliography{bibliography_reduced}

\end{document}